\documentclass[11pt, reqno]{amsart} 
\usepackage{bigints}
\usepackage{todonotes}
\presetkeys{todonotes}{color=orange!30}{}

\usepackage[english]{babel}
\usepackage[babel]{csquotes}

\usepackage{amsmath} 
\usepackage{mathrsfs}
\usepackage{eucal}	
\usepackage{amssymb}

\numberwithin{equation}{section}

\theoremstyle{plain} 
\newtheorem{theorem}{Theorem}[section]
\newtheorem{proposition}[theorem]{Proposition}
\newtheorem{lemma}[theorem]{Lemma}

\newtheorem{corollary}[theorem]{Corollary}

\theoremstyle{definition}
\newtheorem{remark}[theorem]{Remark}
\newtheorem{example}[theorem]{Example}

\newtheorem{definition}[theorem]{Definition}

\newcommand{\be}{\end{eqnarray*}}
\newcommand{\ee}{\end{eqnarray*}}
\newcommand{\ben}{\begin{eqnarray}}
\newcommand{\een}{\end{eqnarray}}

\def \R{\mathbb R}
\def \N{\mathbb N}

\def \T{\mathcal T}

\def \vsm{\vskip 0.2 truecm}
\def \ds{\displaystyle}

\def\bel{\begin{equation}\label}
\def\eeq{\end{equation}}
\def \w{\omega}

\def \weak{\overset{*}{\rightharpoonup}}

 \makeatletter
\@namedef{subjclassname@2020}{%
  \textup{2020} Mathematics Subject Classification}
\makeatother

\begin{document}
\title[Impulsive delay systems]{Optimal impulsive control for time delay systems}
 \thanks{This research is partially supported by the  INdAM-GNAMPA Project, codice  CUP-E53C22001930001.}

\author{Giovanni Fusco}\address{G. Fusco, Department of Mathematics `Tullio Levi-Civita`, University of Padova, Padova, 35121, Italy\\
email:\,
fusco@math.unipd.it}
\author{Monica Motta}\address{M. Motta, Department of Mathematics `Tullio Levi-Civita`, University of Padova, Padova, 35121, Italy\\
email:\,
motta@math.unipd.it}
\author{Richard Vinter}\address{R. Vinter, Department of Electrical and Electronic Engineering, Imperial College London, Exhibition Road, London SW7 2BT, UK\\
email:\,
r.vinter@imperial.ac.uk}

\begin{abstract}
We introduce discontinuous solutions to  nonlinear impulsive control systems with state time delays in the dynamics and derive necessary optimality conditions in the form of a Maximum Principle for associated optimal control problems.  In the case without delays, if the measure control is scalar valued, the corresponding discontinuous state trajectory, understood as a limit of classical state trajectories for absolutely continuous controls approximating the measure, is unique. For vector valued measure controls however, the limiting  trajectory is not unique and a full description of the control must  include additional `attached' controls affecting instantaneous state evolution at a discontinuity.  For impulsive control systems with time delays  we reveal a new phenomenon, namely that the limiting state trajectory resulting from different approximations of a given measure control needs not to be unique, even in the scalar case. Correspondingly, our framework allows for additional attached controls, even though the measure control is scalar valued.
\end{abstract}

\keywords{
Impulsive Control, Time Delay Systems, Optimal Control, Maximum Principle, Bounded Variation.}

\subjclass[2020]{
49N25, 93C43, 49K21, 26A45, 49J21}

\maketitle

\section{Introduction}
This paper provides a well-posed notion of solution 
  for a control system of the form
$$
 {\rm (DIS)} \left\{
\begin{array}{l}
\ds dx(t) =  f(x(t),x(t-h))dt +g(x(t),x(t-h))   d\mu(t) ,\; t \in [0,T] 
\\[1.0ex]
\mbox{in which }  \mu \mbox{ is a non-negative Borel measure on $[0,T]$, and } 
\\[1.0ex] 
x(0)=x_0 \mbox{ and } x(t)= \xi_0 \mbox{ for } t \in [-h,0[\,,
\end{array}
\right.
$$
and necessary conditions of optimality for associated optimal impulsive control problems. 
Here the data  comprise a constant $T >0$, functions $f:\R^n \times \R^n \rightarrow \R^n$ and $g:\R^n \times \R^n \rightarrow \R^n$, vectors $x_0 \in \R^n$ and $\xi_0 \in \R^n$.

There is an extensive literature on control systems of this kind whenever $f$ and $g$ do not depend on the delayed state, both in the  case of scalar and vector valued control measures, which includes versions of the maximum principle for optimal impulsive control problems. Representative papers include \cite{Bressan8,BR88, MR95, MiRu, Sa, SV1, PS, WZ, AKP1, KDPS15, AMR20, MS}. 
The novel feature of the class of problems treated in this paper is that the dynamic constraint contains a time delay. Since first order necessary optimality conditions are available for conventional (`non-impulsive') optimal control problems with time delay (see, e.g, \cite{Boccia} and the references therein), it might be thought that  broadening of the theory  to allow for time delays in the state would be a routine exercise. But, contrary to such expectations, the presence of a time delay gives rise to interesting and unexpected phenomena which, to the authors' knowledge, have  not been previously explored in the literature. Furthermore, the derivation of necessary conditions  poses significant challenges, because the `change of independent variable' techniques employed in the delay free literature, to reduce impulsive optimal control problems to conventional optimal control problems, cannot be adapted directly to the delay setting, owing to the fact that the analogous change of variable does not generate a time-delay optimal control problem of standard form. 

Notice that when the `control' $\mu$ in (DIS) is absolutely continuous w.r.t. Lebesgue measure, i.e. $d\mu(t) =u(t)dt$ for some integrable function $u$, the state trajectory is simply a solution to the   delay differential equation
$$
\dot x(t) = f(x(t), x(t-h)) +g(x(t), x(t-h)) u(t), \mbox{ a.e. } t \in [0,T] 
$$
satisfying the specified left end-point condition.  But if the measure $\mu$ has an atom at $\bar t$, then according to the control system, the state variable may jump at time $\bar t$.

The motivation behind the definition of `solution' to the state equation, in the case without delays,  was first provided by aerospace applications, in which (as in (DIS)) the right  side of the dynamic relation is affine in the control variable and where an `impulsive control' is an idealization of high intensity control action over a time interval of short duration.   The appropriate definition of state trajectory is then that it should be the limit of some sequence of state trajectories associated with a convergent sequence of absolutely continuous controls \cite{Lawden}.

The delay-free impulsive control literature, where the dynamic constraint takes the form
$$
\left\{
\begin{array}{l}
\ds dx(t) = \tilde f (x(t))dt +\tilde g(x(t))  d\mu(t) ,\; t \in [0,T] 
\\[1.0ex]
\ds\mbox{in which }  \mu\mbox{ is a non-negative Borel measure on } [0,T], \mbox{ and } 
\\[1.0ex]
  x(0)=x_0,
\end{array}
\right. \qquad\qquad\qquad
$$
  points the way to defining state trajectories that can be interpreted as limits of non-impulsive state trajectories. For an arbitrary impulsive control $\mu$, all limits of state trajectories corresponding to  absolutely continuous controls  approximating $\mu$ (in a sense specified in the example below) are the same.  If, furthermore,  $\mu(\{\tau\})=\w(\tau)$   at a time $\tau$,  then its state trajectory $x$ has a jump  $y(1) -x(\tau^-)$ at $\tau$, where $y:[0,1] \rightarrow \R^n$ is the solution to the differential equation
$$
\left\{
\begin{array}{l}
 \dot y(s) =   \w(\tau)\tilde g(y(s)), \mbox{ a.e. } s \in [0,1]
 \\[1.0ex]
 y(0) = x(\tau^-)\,.
\end{array}
\right.
 $$
As is well-documented, the property `all sequences of  state trajectories corresponding to some sequence of absolutely controls that approximates the nominal impulsive control $\mu$ have the same limit'  depends critically on the assumption that the impulsive control $\mu$ is scalar valued (see \cite{Bressan8}).
 If the delay-free model is modified to include a $k$-vector valued impulsive control $\mu =(\mu_1,\ldots, \mu_k)$, thus 
$$
\left\{
\begin{array}{l}
\ds dx(t) = \tilde f (x(t))dt +\sum_{i=1}^k\tilde g_i(x(t))   \mu_i(dt) ,\; t \in [0,T] 
\\
\mbox{in which $\mu_i$, $i=1,\ldots,k$, are non-negative Borel measures on } [0,T], \mbox{ and } 
\\[1.0ex]
x(0)=x_0\,,
\end{array}  
\right.
$$
the uniqueness property is lost: different approximations of the vector impulsive controls by absolutely continuous controls can give rise to different state trajectories in the limit, unless we impose rather stringent `commutativity' hypotheses on the $\tilde g_i$'s  (see \cite{BR88,BR91,MR95}). In this situation, the family of state trajectories associated with a given vector valued impulsive control  $\mu$ can be parameterized by collections of control functions attached to each point of discontinuity in the distribution of $\mu$, which determine the evolution of the state during the jump (see \cite{AKP11,KDPS15}).

Now let us return to consideration of impulsive control systems with time delay, when the impulsive control is scalar valued. We might expect, by analogy with the delay-free case, that all sequences of state trajectories associated with some sequence of absolutely continuous controls that approximates a given impulsive control, would have the   same limit and this could be defined as the state trajectory for the nominal impulsive control. But instead we encounter a new phenomenon: the uniqueness property is lost and, to capture the entire possible family of limiting state trajectories, we need to consider collections of attached controls describing the evolution of the state during the jump, in a way formerly encountered only for vector valued impulsive control systems.

\begin{example}\label{example1}
Consider the impulsive control system with scalar state $x$:
$$
\begin{cases}
dx(t) = x(t) x(t-1)   d\mu(t) \qquad \text{a.e. $t \in [0,2]$,} \\[1.0ex]
\text{in which $\mu$ is a non-negative Borel measure on $[0,2]$ and}  \\[1.0ex]
 x(t)=1  \quad \forall t \in [-1,0]\,.
\end{cases}
$$
Take the nominal  measure control to be
$
 \mu= \delta_{\{1/2\}}+ \delta_{\{3/2\}}
$,
where $\delta_{\{\bar t\}}$ denotes the unit Borel measure concentrated at $t=\bar t$.
Take $\varepsilon_i \downarrow 0$ and, for any integer $i\geq1$, let $u_i \, dt$ be the sequence of absolutely continuous measures approximating $\mu$:
$$
u_i(t) =
\begin{cases}
\varepsilon_i^{-1} &\mbox{\hspace{0.3 in}if } t \in [\frac{1}{2}, \frac{1}{2} + \varepsilon_i] \cup [\frac{3}{2}- \varepsilon_i, \frac{3}{2}]
\\
0& \qquad \text{otherwise}.
\end{cases}
$$
(The sequence $(u_i \, dt)$ approximates $\mu$ in the sense that $u_i\,dt \weak  \mu$, in the weak$^*$ $C(0,2)$-topology.)
Write $x_i$ for the corresponding state trajectories. A simple calculation reveals that $x_i(t) \rightarrow x(t)$ for all $t \in [0,2]\setminus \{\frac{1}{2},\frac{3}{2}\}$, where $x(t)=1$ if $t\in[0,1/2[$, $x(t)=e$ if $t\in[1/2,3/2[$ and $x(t)=e^2$ if $t\in[3/2,2]$.
Now let  $(\tilde u_i\, dt)$  be  a different sequence of absolutely continuous measures approximating  the nominal measure control $\mu$   given by
$$
\tilde u_i(t) =
\begin{cases}
\varepsilon_i^{-1} &\mbox{\hspace{0.3 in}if } t \in [\frac{1}{2} - \varepsilon_i, \frac{1}{2}] \cup [\frac{3}{2}, \frac{3}{2}+ \varepsilon_i]
\\
0& \mbox{\hspace{0.3 in}otherwise}.
\end{cases}
$$
Write $\tilde x_i$ for the corresponding state trajectories. We find that $\tilde x_i(t) \rightarrow \tilde x(t)$ for all $t \in [0,2]\setminus \{\frac{1}{2},\frac{3}{2}\}$, where $\tilde x(t)=1$ if $t\in[0,1/2[$, $\tilde x(t)=e$ if $t\in[1/2,3/2[$ and $\tilde x(t)=e^{1+ e}$ if $t\in[3/2,2]$.
Notice that the arcs $x$ and $\tilde x$ differ on the subinterval $[\frac{3}{2}, 2]$.
\end{example}
In this paper we provide, for the first time, a framework for studying the multiplicity of state trajectories associated with a given impulsive control, in the presence of time delays. We pose an optimal control problem where selection of a state trajectory corresponding to the chosen impulsive control is part of the optimization procedure. We give conditions for existence of solutions to this control problem and provide necessary optimality conditions in the form of a maximum principle.

Let us briefly discuss the methodological challenges of generalizing the techniques previously employed for studying impulsive control systems, when we  introduce time delays.  Take a measure control $\mu$ with associated state trajectory $x$.  The key idea in earlier work concerning delay-free systems, was to introduce a measure control dependent change of independent variable $s= \psi (t)$, satisfying $\psi(0)=0$ and $\psi(t)=t+\mu([0,t])$ for $t \in ]0,T]$,
and interpret a discontinuous  state trajectory $x$ as an absolutely continuous arc $y$, under the discontinuous change of independent variable, thus $x(t)= y(\psi (t))$. The discontinuity  in $x$ at time $t$ is captured by the evolution of $y$ on $[\psi(t^-), \psi(t^+)$] (the interval between left and right limits of $\psi$ at $t$). Each $y$ is governed by a controlled differential equation. The `controls' in this new system description generate all possible state trajectories for the original control system with measure control $\mu$, obtainable as appropriate limits of absolutely continuous state trajectories.  We use this property to study the set of impulsive state trajectories via a conventional control system. We can, for example, derive necessary conditions of optimality for  impulsive control problems in this way.

This approach cannot  be directly applied, however,  when we introduce time delays, because  the control system description that results from the above change of variable is a  controlled delay differential equation, {\it in which the time delays vary with time and are control dependent.} This complicated delay differential equation is not amenable to analysis.
Instead, we use a more subtle change of independent variable that is, in a sense, uniform over different time segments of length $h$ and which preserves constancy of the  delays under transformation.

Research is currently in progress into various extensions of this paper. These include generalizations to allow for  end-times that are not integer multiples of a single time delay, vector-valued measure controls with  range a convex cone, and for non-autonomous control systems with multiple, commensurate time delays. Conditions under which the infimum cost over classical controls is the same as that over measure controls, so called `gap conditions',  are also under investigation (see e.g. \cite{MRV, FM1, FM2} in the case without delays).

  It is worth mentioning that in the literature there are several results on the stabilizability of delayed impulsive control systems and on the optimization of some  specific related problems, but they all concern the so-called ‘impulse model’, where impulsive controls essentially reduce to a finite or countable number of jump instants, with a preassigned jump-function. The notion of impulsive control system introduced in the present paper, which includes the use of strategies with a significant freer allocation of the impulses,  might allow the development of new nonlinear models, for instance, of  fed-batch fermentation  \cite{XSSZ02,GLFX06} or of delayed neural networks  with impulsive controls \cite{LCH20}. We also mention \cite{FM23}, in which necessary optimality conditions are introduced for an impulsive control  problem with (possibly non-commensurate) time delays in the drift term alone of the dynamics.

The paper is structured as follows. In Sec. \ref{sec2} we give an appropriate definition of solution to (DIS), in Sec. \ref{sec3} we introduce an ordinary control system which can be obtained by (DIS) by means of a suitable change of independent variable, in Sec. \ref{sec4} we establish the main properties of solutions to (DIS), while in Sec. \ref{S_PMP} we associate with (DIS) a Mayer cost and an endpoint constraint and we prove a maximum principle for this optimization problem. 


 \subsection{Notation}\label{sub1.1}
Given a set $X \subseteq \R^k$, we write  $W^{1,1}([a,b];X)$,  
 $L^1([a,b];X)$, $L^\infty ([a,b];X)$, $BV([a,b];X)$, for the set of  absolutely continuous,  
 Lebesgue integrable, essentially bounded, and bounded variation functions 
 defined on $[a,b]$ and with values in $X$, respectively. We will not specify domains and  codomains when the meaning is clear  and  we will  use $\| \cdot \|_{L^{\infty}(a,b)}$,  or also   $\| \cdot \|_{L^{\infty}}$,  to denote the ess-sup norm.  As usual, given $T>0$,  $S>0$ and an increasing, surjective map $\sigma:[0,T] \to [0,S]$, $\sigma^-(t)$ and  $\sigma^+(t)$ denote the left and right limits of $\sigma$ at $t$, respectively, when $t \in (0,T)$, while  $\sigma^-(0):=0$ and $\sigma^+(T):=S$. Note that we call increasing any monotone non decreasing function.
 We denote by $C^\oplus(0,T)$ the set of Borel non-negative scalar valued measures on $[0,T]$ (from now on we will refer to such $\mu$ simply as measures). For any $\mu\in C^\oplus(0,T)$, we use both the notations $\|\mu\|_{TV}$ and $\int_{[0,T]}d\mu(t)$ for the total variation of $\mu$.
In the following,    {\em $\mu$-a.e.} means ``almost everywhere w.r.t. $\mu$", and when we do not specify $\mu$ we implicitly refer to  Lebesgue measure.
Given a sequence $(\mu_i)\subset  C^\oplus(0,T)$ and $\mu\in  C^\oplus(0,T)$, as customary we write $\mu_i\weak \mu$ if  
$\lim_i \int_{[0,T]} \Phi(t) \mu_i(dt) = \int_{[0,T]} \Phi(t) d\mu(t)$ for all  continuous functions $\Phi:[0,T]\to\R$.

\noindent
The \textit{limiting normal cone} $N_\T(\bar x)$ to a closed set $\T \subseteq \R^k$ at $\bar x\in\R^k$  is   
 \[ 
N_\T(\bar x) := \left\{  \eta\in\R^k  \text{ : } \exists x_i \stackrel{\T}{\to} \bar x,\, \eta_i \to \eta \,\,\text{ s.\,t. }\,\,  \limsup_{x \to x_i} \frac{\eta_i \cdot (x-x_i)}{|x-x_i|}  \leq 0   \ \ \forall i     \right\},
\]
 in which the notation $x_{i} \stackrel{\T}{\longrightarrow}\bar{x}$ means that   $(x_i)_{i}\subset \T$.

\section{Impulsive Controls and Extended Processes} \label{sec2}
Consider the control system
$$
{\rm (DIS)}\left\{
\begin{array}{l}
\ds dx(t) = f(x(t), x(t-h))dt +g(x(t), x(t-h))  d\mu(t) ,\; t \in [0,T] 
\\[1.0ex]
\mbox{ for some }  \mu \in C^{\oplus} (0,T) \mbox{ and } x \in BV([-h,T];\R^n)  \mbox{  satisfying } 
\\[1.0ex]
x(0)=x_0 \mbox{ and } x(t)= \xi_0 \mbox{ for } t \in[-h,0[\,,
\end{array}
\right.
$$
where $f,g: \R^n \times\R^n  \rightarrow \R^n$,  $x_0$  and  $\xi_0  \in \R^n$,  and  $h >0$, $T >0$. 
Assume:
\vsm

\noindent
(H1): $f$ and $g$ are $C^1$ and bounded functions.
\vspace{0.05 in}

\noindent
(H2): $T =Nh, \mbox{ for some integer } N >0$.

\begin{definition}[Impulsive and Strict Sense Controls]\label{imp_control_def} 
An {\em impulsive control}  $(\mu,  \{w^r\}_{ r \in [0,h]})$ comprises a function $\mu \in  C^{\oplus}(0,T)$ and a family of measurable functions $w^r =(w^r_1,\ldots, w^r_N ) : [0,1] \rightarrow [0,+\infty[^N$, parameterized by $r \in [0,h]$, with the following properties:
\begin{itemize}
\item[{\rm (i)}:] For any $r \in [0,h]$, 
$
\ds\sum_{i=1}^N w^r_i (s) =  \sum_{i=1}^N \int_0^1 w^r_i (s) ds, \ \mbox{ a.e. } s \in [0,1];
$ 
\item[{\rm (ii)}:] For any $r \in ]0,h[$, 
$
\ds\int_0^1 w^r_i (s) ds = \mu (\{r+(i-1)h\}), \ \ \mbox{for }i=1,\dots,N.
$  
\item[{\rm (iii)}:] $
 \ds\int_0^1 [w^h_i (s)+w^0_{i+1}(s)] ds = \mu(\{ih\}), \ \ \mbox{for }i=0,\dots,N
$,  where   $w^h_0\equiv0$ and  $w^0_{N+1}\equiv0$.
\end{itemize}
A {\em strict sense control} is an impulsive control $(\mu, \{w^r\}_{ r \in [0,h]})$, in which $ d\mu(t) = u(t)dt$ for some nonnegative function $u \in L^1 (0,T)$. With a slight abuse of notation, we also call $u$   strict sense control. (Notice that, for a strict sense control,  for any $r \in [0,h]$ and $i=1,\dots,N$, we have $w^r_i (s)=0$ a.e. since $d\mu(t) = u(t)dt$ has no atoms.)  
\end{definition}

%

\begin{definition}[Extended  and Strict Sense Processes]\label{def_ext}
We say that a triple $(x, \mu, \{w^{r}\}_{ r \in [0,h]})$ is  an {\em extended process} (with {\em extended state trajectory} $x$) if $x \in BV([-h,T]; \R^n)$ and $(\mu, \{w^{r}\}_{ r \in [0,h]})$ is an impulsive control such that
\vspace{0.05 in}

\noindent
{\rm (i)}:  $x(t)=\xi_0$, for $t\in[-h,0[$;
\vspace{0.05 in} 

\noindent 
(ii): for each $i =1,\ldots,N$,
$
x(t)= 
\begin{cases}
x_0 \qquad \mbox{if } i=1\\
\zeta_i^0 (1) \qquad \mbox{if } i >1
\end{cases}
$ for $t =(i-1)h,$
\begin{equation*}
\begin{split}
x(t) =  \zeta^0_i (1) &+\int_{(i-1)h}^t  f(x(t'), x(t'-h))dt'   
 + \int_{[(i-1)h,t]} g(x(t'), x(t'-h))d\mu^{c}(t') \\
& + \underset{r \in ]0,\, t- (i-1)h] }{\sum} ( \zeta^r_i (1) - x((r +(i-1)h)^-) 
 \qquad \mbox{ for } t \in ](i-1)h, ih[
 \end{split}
\end{equation*}
and  $x(T)= \zeta^h_N (1)$.
Here, for $r \in [0,h]$,  the functions $\zeta^r_1,\ldots, \zeta^r_N:[0,1]\rightarrow \R^n$ satisfy  the coupled differential equations
\begin{equation}
\label{jump_ode1}
\frac{d}{ds} \zeta^r_i (s)  \,=\,w_i^{r}(s)\, g(\zeta^r_i (s), \zeta^{r}_{i-1} (s) ), \mbox{ a.e. } s \in [0,1],\quad  \mbox{ for } i =1,\ldots, N\,,
\end{equation}
with boundary conditions
\begin{equation}
\label{jump_ode2}
\zeta^r_i (0) =
\left\{
\begin{array}{ll}
x^-(r + (i-1)h) & \mbox{if } r \in ]0,h]
\\
\zeta_{i-1}^h (1)&\mbox{if } r =0 \,.
\end{array}
\right.
\end{equation}
In these relations,  
$
\zeta^r_0(s) :=  \left\{
\begin{array}{ll}
\xi_0 &\mbox{ if  $s\in [0,1[$ or $r\in[0,h[$}
\\
x_0 &\mbox{if $s=1$ and $r=h$}
\end{array}
\right.$ 
and $\mu^{c}$ denotes the continuous component of the measure $\mu$. (Extended state trajectories have bounded variation and are right continuous on $]0,T[$.)

\noindent
An extended state process $(x, \mu,  \{w^{r}\}_{ r \in [0,h]})$ corresponding to a strict sense control $u$, also written simply $(x,u)$, is called a {\em  strict sense process}. In these circumstances  $x$, referred to as a strict sense state trajectory, is an absolutely continuous function that satisfies the delayed differential equation
$$
\dot x (t) = f(x(t), x(t-h)) + g(x(t), x(t-h) )u(t),\mbox{  a.e. }t \in [0,T]\,.
$$
 \end{definition}

Note that if, for some $i=1,\dots, N$, $t\in](i-1)h,ih[$ is not an atom of $\mu$, then $w^{t-(i-1)h}_i=0$ a.e  and the trajectory $x$ is continuous at $t$, as the function $\zeta^{t-(i-1)h}_i$  describing the jump of the state   at the instant $t$, turns out to be constant on $[0,1]$.
On the other hand, when $t=ih$ for $i=1,\dots,N-1$, the possible atom of the measure $\mu$ at $t$ (hence, the corresponding jump of the trajectory $x$) is determined by the combined effect  of both the attached controls $w^h_i$ and $w^0_{i+1}$, due to the presence of the time delay $h$.
 
\begin{remark}
In the earlier Example \ref{example1} (in which $h=1$, $N=2$, and $T = Nh=2$), we showed how different approximations of a measure control $\mu$ by absolutely continuous controls can give rise to different state trajectories. We show how these observations are consistent with the preceding definition of state trajectories corresponding to impulsive controls.  
Choose once again
$
\mu =  \delta_{\{1/2\}} +  \delta_{\{3/2\}}\,.  
$
Now consider the following impulsive control $(\mu,  \{w^r \}_{r \in [0,1]})$, in which  $w^r \equiv 0$  if $r \not= \frac{1}{2}$ and
$$
w^{r = \frac{1}{2}}_1(t) =
\begin{cases}
0 & \mbox{if } 0 \leq t \leq \frac{1}{2}
\\
2 & \mbox{if } \frac{1}{2} <t  \leq 1
\end{cases}
\qquad \mbox{ and } \qquad
w^{r = \frac{1}{2}}_2(t) =
\begin{cases}
2 & \mbox{if } 0 \leq t \leq \frac{1}{2}
\\
0 & \mbox{if } \frac{1}{2} <t  \leq 1.
\end{cases}
$$
Furthermore, consider $(\mu, \{\tilde{w}^r \}_{r \in [0,1]})$  in which   $\tilde w^r \equiv 0$, if $r \not= \frac{1}{2}$ and
$$
\tilde w^{r = \frac{1}{2}}_1(t) =
\begin{cases}
2 & \mbox{if } 0 \leq t \leq \frac{1}{2}
\\
0 & \mbox{if } \frac{1}{2} <t  \leq 1
\end{cases}
\qquad \mbox{ and } \qquad
\tilde w^{r = \frac{1}{2}}_2(t) =
\begin{cases}
0 & \mbox{if } 0 \leq t \leq \frac{1}{2}
\\
 2 & \mbox{if } \frac{1}{2} <t  \leq 1.
\end{cases}
$$
Based on the constructions implicit in the preceding definition, the state trajectories $x$ and $\tilde x$, corresponding to the impulsive controls $(\mu, \{w^r \}_{r \in [0,1]})$ and
$( \mu, \{\tilde{w}^r \}_{r \in [0,1]})$, respectively,  are
\[
 x(t) =
\begin{cases}
1   &\mbox{if } t \in [0, \frac{1}{2}[
\\
e   &\mbox{if } t \in [ \frac{1}{2}, \frac{3}{2}[
\\
e^2   &\mbox{if } t \in [\frac{3}{2}, 2]
\end{cases}
\quad \text{and} \quad
\tilde x(t) =
\begin{cases}
1   &\mbox{if } t \in [0, \frac{1}{2}[
\\
e   &\mbox{if } t \in [ \frac{1}{2}, \frac{3}{2}[
\\
e^{1+ e}   &\mbox{if } t \in [\frac{3}{2}, 2].
\end{cases}
\]
These formulae reproduce the two distinct limits of state trajectories corresponding to absolutely continuous controls, identified in the earlier analysis of this example. 
\end{remark}

\section{Reparameterization}\label{sec3}

Processes for {\rm (DIS)} will be analysed with the help of a change of independent variable that replaces possibly discontinuous state trajectories by continuous functions.  The new set of processes that arise in this way are called reparameterized processes for {\rm (DIS)}.
\begin{definition}\label{def_rep}
A {\em reparameterized process} for {\rm (DIS)} is a collection  of elements $(S, \{y_i\}, \{\alpha_i\})$, comprising a number $S >0$,  controls $\alpha_i \in L^\infty (0,S)$, $i=1,\ldots, N$, and functions $y_i \in W^{1,1}([0,S];\R^n)$, $i= 0,\ldots, N$,    that satisfy
$$
\ \left\{
\begin{array}{l}
\ds\dot y_i (s) =\Big(1- \sum_{i=1}^{N} \alpha_i (s)\Big) f(y_i (s),y_{i-1}(s)) + \alpha_i (s)g( y_i(s), y_{i-1}(s)), 
\\
\hspace{2.5 in} \mbox{ a.e.  $s \in [0,S]$, for $i=1,\ldots, N$}, 
\\
\ds\alpha_i(s) \geq 0 \mbox{ for each } i=1,\dots,N \mbox{ and } \sum_{i=1}^{N} \alpha_i(s) \leq 1, \mbox{ a.e. } s\in [0,S] ,
\\[1.0ex]
 y_0(s)=\xi_0 \ \text{for  $s\in[0,S[$}, \ y_0(S)=x_0, \ \     y_{i+1}(0) =y_{i}(S), \mbox{ for } i=0,\ldots N-1\,, 
\\ [1.1ex]
\ds \int_0^S  \Big(1- \sum_{i=1}^{N} \alpha_i (s)\Big)ds =h .
\end{array}
\right. 
$$
\end{definition}
 \begin{remark}\label{rem_uniq}
Because $f$ and  $g$ are $C^1$,  bounded functions, for any $S>0$ and all controls $\alpha_1,\dots,\alpha_N\in L^1(0,S)$  such that $\alpha_i(s) \geq 0$ for each  $i$,  $\sum_{i=1}^{N} \alpha_i(s) \leq 1$  a.e., and  $\int_0^S  (1- \sum_{i=1}^{N} \alpha_i (s))ds =h$, there exists a unique  set  of functions $y_1,\dots,y_N\in   W^{1,1}([0,S];\R^n)$ such that $(S, \{y_i\}, \{\alpha_i\})$ is a reparameterized process for {\rm (DIS)}.
\end{remark}

\begin{theorem} 
\label{Prop_3_2}
The following statements hold true:
\begin{itemize}
\item[(A):] 
Let $(S , \{y_i\}, \{\alpha_i\})$ be a reparameterized process for {\rm (DIS)}. Then, there exists an extended process 
$(x,\mu,  \{w^{r}\}_{ r \in [0,h]})$ for  {\rm (DIS)} such that $x(T)=y_N (S)$ and
\bel{TV}
\|\mu\|_{\mbox{TV}} ={\sum}_{i=1}^N\; \int_0^S \alpha_i (s)ds.
\eeq
\item[(B):]
Let $(x,\mu, \{w^{r}\}_{ r \in [0,h]})$ be an extended process for  {\rm (DIS)}. Then there exists a reparameterized  process   $(S,\{y_i\},\{\alpha_i\})$ for  {\rm (DIS)} such that $y_N (S)= x(T)$  and \eqref{TV} is satisfied.
\end{itemize}
\end{theorem}

 Proof of this theorem requires the following lemma. 
\begin{lemma}
\label{lem_3_3}
Take an increasing function $A:[0,a]\to[0,b]$ such that $A(0)=0$, $A(a)=b$, and $A$ is right continuous on $]0,a[$. Define its {\em right inverse} as the function $B:[0,b]\to[0,a]$ such that
$$
B(0):=0, \qquad B(r):=\inf\{s\in[0,a]: \ \ A(s)>r\}    \ \text{for } r\in]0,b[, \qquad B(b):=a.
$$
Then, $B$ is increasing,  right continuous on $]0,b[$, and it is continuous iff $A$ is strictly increasing.  For any $r\in[0,b]$ one has  $A(B(r))\ge r$  and, when $A$ is continuous at  $B(r)$,    $A(B(r))= r$.  Moreover,   $A$ is the right inverse of $B$, namely
$$
A(s)=\inf\{r\in[0,b]: \ \ B(r)>s\}    \ \text{for } s\in]0,a[.
$$
Furthermore, if $F:[0,a]\to\R^n$ is a Borel function, then
$$
\int_{[B(r_1),B(r_2)]}F(s) dA(s)=\int_{[r_1,r_2]}F(B(r)) dr   \quad\text{for every $0\le r_1<r_2\le b$,}
$$
where $dA(s)$ is the measure associated with the increasing function $A$.   
\end{lemma}
\begin{proof} If we consider $A$ as the restriction to $[0,a]$ of an increasing and right continuous function, continuous outside $[0,a]$, the properties of $B$ follow from \cite[Lemma 4.8,\,Ch.\,0]{RY}.  
 For each $\ell=1,\dots,n$, let $F^+_\ell$ and $F^-_\ell$ be the positive and the negative part of the component $F_\ell$ of $F$, so that $F^+_\ell$,   $F^-_\ell\ge0$, and $F_\ell(s)=F^+_\ell(s)-F^-_\ell(s)$ for all $s\in[0,a]$. Fix $0\le r_1<r_2\le b$. Then, if we extend $F^+_\ell$,   $F^-_\ell$   as nonnegative Borel functions on $[0,+\infty[$,  \cite[Prop. 4.9,\,Ch.\,0]{RY} implies that
$$
\int_{[0,+\infty[}F^\pm_\ell(s)\chi_{[B(r_1),B(r_2)]}(s)dA(s)=\int_{[0,+\infty[}F^\pm_\ell(B(r))\chi_{[r_1,r_2]}(r)dr,
$$
from which we immediately derive
$
\int_{[B(r_1),B(r_2)]}F_\ell(s) dA(s)=\int_{[r_1,r_2]}F_\ell(B(r)) dr,
$
 for each  value of the index $\ell$.  
\end{proof}

\begin{proof}[ Proof of  Thm. \ref{Prop_3_2}]

\noindent
{\bf (A):}
Take a reparameterized process  $(S, \{y_i\}, \{\alpha_i\})$  for {\rm (DIS)}. Define the mapping $\psi:[0,S] \rightarrow \R$ to be
\bel{psi}
\psi(s) :=   \int_0^s \Big(1 -  \sum_{i =1}^{N}  \alpha_i (s') \Big)d s', \; \mbox{for } s \in [0,S]\,.
\eeq
 Directly from the definition of reparameterized processes for  {\rm (DIS)},  we know that $\psi(S)=h$.  Because $\psi$ is a   continuous,  increasing function,  by Lemma \ref{lem_3_3} its right inverse function $\sigma:[0,h]\to[0,S]$,  given by
$$
\sigma(0):=0, \qquad \sigma(r):=\inf\{s\in[0,S]: \ \ \psi(s)>r\}  \ \text{for } r\in]0,h[, \qquad \sigma(h):=S,
$$
 is strictly increasing, right continuous on $]0,h[$, and satisfies
$$
\sigma(r)\ge r, \quad  \psi(\sigma(r))=r, \quad\text{for any $r\in[0,h]$.}
$$
Moreover, it has at most a countable set  $(r_k)_{k\in\N}\subset[0,h]$ of discontinuity points, where $\sigma^- (r_k)<\sigma^+ (r_k)$, 
  which correspond to the constancy intervals of $\psi$, namely 
  $
 \psi^{-1}(\{r_k\}) = [\sigma^- (r_k), \sigma^+ (r_k)] 
 $
 for all $k\in\N$.
 Notice that, by the definition of $\psi$, for any $r\in[0,h]$, 
\begin{equation}
\label{2_1}
 \sum_{j=1}^N \alpha_j(s)=1, \mbox{ for a.e. } s \in [\sigma^- (r), \sigma^+ (r)].
  \end{equation}
(If $r\notin (r_k)_k$, we trivially have $[\sigma^- (r), \sigma^+ (r)]=\{\sigma (r)\}$.) Integrating across this relation over the subinterval $[\sigma^- (r), \sigma^+  (r)]$,  yields
 \begin{equation}\label{2_2}
 \sigma^+ (r) - \sigma^- (r) =  \sum_{j=1}^N \int_{\sigma^- (r)}^{\sigma^+ (r)}  \alpha_j (s)\,ds, \quad\text{for any $r\in[0,h]$.}
 \end{equation}
Let us also introduce the continuous,  increasing function $\tilde\psi:[0,NS]\to\R$,  obtained by an $N$-fold time advance and concatenation of  $\psi$, namely
 $$
 \tilde\psi(s):= (i-1)h+\psi(s-(i-1)S), \quad \text{for $s\in[(i-1)S,iS]$,} \quad i=1,\dots,N, 
 $$
 and its right inverse $\tilde\sigma:[0,T]\to[0,NS]$. 
Note that
\begin{equation}\label{2_2_0tilde}
 \tilde\sigma^- (r)= \sigma^- (r-(i-1)h)+(i-1)S, \quad \tilde\sigma^+ (r)= \sigma^+ (r-(i-1)h)+(i-1)S,
  \end{equation}
for any $i=1,\dots,N$ and  $r\in](i-1)h,ih[$, while 
  \begin{equation}\label{2_2_0ih}
  \begin{array}{l}
  \tilde\sigma^- (0)= \sigma^- (0)=0, \quad \tilde\sigma^+ (0)= \sigma^+ (0), \\[1.5ex]
   \tilde\sigma^- (ih)= \sigma^- (h)+(i-1)S,  \quad \tilde\sigma^+(ih) = \sigma^+ (0)+iS, \quad i=1,\dots,N-1,
    \\[1.5ex]
    \tilde\sigma^- (Nh)= \sigma^- (h)+(N-1)S, \quad \tilde\sigma^+ (Nh)= \sigma^+ (h)+(N-1)S=NS.
 \end{array}
  \end{equation}
 
We now construct an extended process $(x,\mu,  \{w^{r}\}_{ r \in [0,h]})$ for (DIS) with the desired properties. 
Consider the   functions $\tilde y\in W^{1,1}([0,NS];\R^n)$ and $\tilde\alpha\in L^1(0,NS)$, given by
 \begin{equation}\label{tilde_y_alpha}
 \begin{array}{l}
\tilde y(s)=y_i(s-(i-1)S),\quad \text{for $s\in[(i-1)S,iS]$, \  $i=1,\dots,N$} \\[1.5ex]
 \tilde \alpha(s)=\alpha_i(s-(i-1)S), \quad\text{a.e. $s\in[(i-1)S,iS[$, \ $i=1,\dots,N$.}
 \end{array}
 \end{equation}
Define $\mu \in C^{\oplus}(0,T)$ and $x:[0,T]\to\R^n$ according to the following relations:   
 \begin{equation}
\label{3}
 \mu([0,t]) :=    \int_0^{\tilde\sigma^+ (t)} \tilde\alpha(s)\, ds \quad\text{for any  $t\in [0, T]$,}
  \end{equation}
and 
 $
 x(t) :=   \tilde y(\tilde\sigma (t))
 $
for any  $t\in [0, T]$, $x(t)=\xi_0$ for $t<0$  (notice that $\tilde\sigma^+(t)=\tilde\sigma(t)$ for all $t\in]0,T]$).
Finally, take any $r \in [0,h]$. For all $i=1,\dots,N$, define $w_i^r:[0,1] \rightarrow [0,+\infty[$  and $\zeta^r_i : [0,1] \rightarrow  \R^n$ as follows: 
\begin{equation*}
\begin{split}
&w_i^r (s) := (\sigma^+ (r) - \sigma^- (r))  \alpha_i (\sigma^-(r) + s\,(\sigma^+ (r) - \sigma^- (r))) , \; \mbox{ a.e. } s \in [0,1],\\
&\zeta_i^r (s) := 
  y_i (\sigma^-(r) + s\,(\sigma^+ (r) - \sigma^- (r) )), \; \mbox{ for all } s \in [0,1].
\end{split}
\end{equation*}
Moreover, set  $\zeta_0^r (s):=\xi_0$ if $s\in[0,1[$ or $r\in[0,h[$ and $\zeta_0^{h} (1):=x_0$.
From \eqref{2_1} it follows that
$
 \sum_{i=1}^N w^r_i (s) = (\sigma^+ (r) - \sigma^- (r)  ) =  \sum_{i=1}^N \int_{\sigma^- (r)}^{\sigma^+ (r)}  \alpha_i (s)ds$ for a.e. $s \in [0,1]$. 

By changing the independent variable in the  integral, we see that
\bel{int_w}
\int_0^1 w^r_i (s) ds 
  =
 \int_{\sigma^- (r)}^{\sigma^+ (r)}\alpha_i (s)ds \quad\text{for $i=1,\dots,N$}.
\eeq
Therefore, condition (i) in Definition \ref{imp_control_def} of impulsive control is satisfied.  It can be deduced from  \eqref{3} and \eqref{2_2_0tilde} that,  for  $r\in]0,h[$ and $i=1,\dots,N$,
$$
\mu(\{r+(i-1)h\})=
 \int_{\sigma^-(r)+(i-1)S}^{\sigma^+(r)+(i-1)S}\alpha_i(s-(i-1)S)\,ds=\int_{\sigma^-(r)}^{\sigma^+(r)}\alpha_i(s)\,ds.
$$
Together with  \eqref{int_w}, this relation implies condition (ii) in Definition \ref{imp_control_def}.  Finally,  condition (iii) follows from   \eqref{3}, \eqref{2_2_0ih}, and \eqref{int_w}.  Indeed
$
\mu(\{0\})= \int_{\sigma^-(0)}^{\sigma^+(0)}\alpha_1(s)\,ds=\int_0^1 w^0_1 (s) ds$, 
$\mu(\{Nh\})= \int_{\sigma^-(h)}^{\sigma^+(h)}\alpha_N(s)\,ds=\int_0^1 w^h_N (s) ds$, 
and, for $i=1,\dots,N-1$, $\mu(\{ih\})=
 \int_{\tilde\sigma^-(ih)}^{\tilde\sigma^+(ih)}\tilde \alpha(s)\,ds$, so that:
\begin{equation*}
\begin{split}
\mu(\{ih\})&= \int_{\sigma^-(h)+(i-1)S}^{S+(i-1)S}\alpha_i(s-(i-1)S)\,ds+\int_{iS}^{\sigma^+(0)+iS}\alpha_{i+1}(s-iS)\,ds \\
&= \int_{\sigma^-(h)}^{\sigma^+(h)}\alpha_i(s)\,ds+\int_{\sigma^-(0)}^{\sigma^+(0)}\alpha_{i+1}(s)\,ds=\int_0^1 w^h_i (s) ds+\int_0^1 w^0_{i+1}(s) ds.
 \end{split}
\end{equation*}
Thus,  $(\mu   ,\{w^r\}_{ r \in [0,h]})$ is an impulsive control. Notice that, from \eqref{3}, 
$$
||\mu||_{TV} = \sum_{i=1}^N \,\int_{(i-1)S}^{iS} \alpha_i(s-(i-1)S)\,ds= \sum_{i=1}^N \,\int_{0}^{S} \alpha_i(s)\,ds.
$$ 
Take $r \in [0,h]$. Observe that, for $i=1,\dots,N$   the $\zeta^r_i$'s  are absolutely continuous functions such that  
$\frac{d}{ds} \zeta^r_i (s)  \,=\,w_i^{r}(s)\, g(\zeta^r_i (s), \zeta^{r}_{i-1} (s) )$ a.e.  $s \in [0,1]$
and  
$
\zeta^r_i (0) =
y_i( \sigma^- (r)).
$ 
 Furthermore, 
\begin{equation}
\label{y_and_zeta}
y_i(\sigma^+ (r)) -y_i(\sigma^- (r))  = \zeta_i^r (1) -\zeta^r_i (0)\,, \quad\text{for every   $i=1,\ldots, N$.}
\end{equation}
 Consider now $t\in](i-1)h,ih[$ for some $i=1,\ldots, N$.   Then   $\tilde\sigma^+(t)\in ](i-1)S,iS[$ and from the previous calculations it follows that
\begin{equation}
\label{tilde_y_and_zeta}
\begin{split}
\tilde y(\tilde\sigma^+ (t)) -\tilde y(\tilde\sigma^- (t))  &=y_i(\sigma^+ (t-(i-1)h)) -y_i(\sigma^- (t-(i-1)h)) \\
&=\zeta_i^{t-(i-1)h} (1) -\zeta_i^{t-(i-1)h} (0)\,.
\end{split}
\end{equation}
If instead $r=ih$  for some $i=1,\ldots, N$,  \eqref{tilde_y_alpha}  (together with the fact that $y_{i+1}(0) =y_{i}(S)$) yield
\begin{equation}
\label{tildeh_y_and_zeta}
\begin{split}
&\tilde y(\tilde\sigma^+ (ih)) -\tilde y(\tilde\sigma^- (ih))  =\tilde y(\sigma^+ (0)+iS) -\tilde y(\sigma^- (h)+(i-1)S) 
  \\
&\qquad 
=[ y_{i+1}(\sigma^+ (0)) -y_{i+1}(\sigma^- (0))]+[ y_i(\sigma^+ (h)) -y_i(\sigma^- (h))] 
\\
&\qquad= [\zeta_{i+1}^{0} (1) -\zeta_{i+1}^{0} (0)]+
 [\zeta_i^{h} (1) -\zeta_i^{h} (0)]\,,  \quad\text{for $i=1,\dots,N-1$,} \\
 &\tilde y(\tilde\sigma^+ (Nh)) -\tilde y(\tilde\sigma^- (Nh))  = y_N(\sigma^+ (h)) -y_N(\sigma^- (h))=\zeta_N^h (1) -\zeta^h_N (0)\,.
\end{split}
\end{equation}
 For each $t\in ](i-1)h,ih  [$ and $i =1,\dots, N$,    by \eqref{tilde_y_alpha} and \eqref{2_2_0tilde},  we have
\begin{equation*}
\begin{split}
x(t) =\tilde y(\tilde\sigma (t)) =\tilde y (\tilde\sigma^+((i-1)h))&+   \int_{\tilde\sigma^+((i-1)h)}^{\tilde \sigma^+ (t)} f(\tilde y(s), \tilde y(s-S))\frac{d\tilde\psi}{ds}(s)\,ds 
\\
& +  \int_{\tilde\sigma^+((i-1)h)}^{\tilde\sigma^+ (t)} g(\tilde y(s), \tilde y(s-S))\tilde \alpha (s) ds\,,
\end{split}
\end{equation*}
where $\tilde y (\tilde\sigma^+((i-1)h))=\zeta_i^0(1)$.
We deduce from Lemma \ref{lem_3_3} 
 and \eqref{2_2_0tilde}
   that
\begin{equation*}
\begin{split}
\int_{\tilde\sigma^+((i-1)h)}^{\tilde \sigma^+ (t)} f(\tilde y(s), \tilde y(s-S))\frac{d\tilde\psi}{ds}(s)\,ds   
&= \int_{(i-1)h}^{ t} f(\tilde y(\tilde\sigma^+(t')), \tilde y(\tilde\sigma^+(t')-S))\,dt'   \\
 &=     \int_{(i-1)h}^{ t}  f(x(t'), x(t'-h))dt'\,.
\end{split}
\end{equation*}
Define the set of nondegenerate intervals corresponding to the discontinuity points $(t_k)_{k\in\N}$ of $\tilde\sigma$ on $[0,T]$, namely $D:=\cup_{k\in\N}[\tilde\sigma^-(t_k),\tilde\sigma^+(t_k)]\subset[0,NS]$.  

\noindent
Write 
$
\tilde\alpha^0(s)=
\left\{
\begin{array}{ll} 
0 & \mbox{if } s \in D
\\
\tilde\alpha (s)  &\mbox{if } s \notin D
\end{array}
\right.
$.
The continuous component $\mu^{c}$ of  $\mu$ is the measure  with distribution function
$
\mu^{c} ([0,t]) = \int_0^{\tilde\sigma^+ (t)} \tilde\alpha^0(s')ds' 
$
for $t\in[0,T]$.
 Notice that, since $\tilde\sigma(\tilde\psi(s'))=s'$ for all $s'\notin D$,  from Lemma \ref{lem_3_3} (applied to $A=\tilde\sigma$, $B=\tilde\psi$,  and $F(t')=\tilde\alpha^0(\tilde\sigma(t'))$) for $t\in[0,T]$ it follows that
\bel{mu_cont_2}
\int_{[0,t]=[\tilde\psi(0),\tilde\psi(\tilde\sigma(t))]} \tilde\alpha^0(\tilde\sigma(t'))\,d\tilde\sigma(t') = \int_0^{\tilde\sigma (t)} \tilde\alpha^0(s')ds'=\mu^{c} ([0,t]).
\eeq
For each $t \in ](i-1)h,ih [$ and $i =1,\dots, N$, we have
\begin{equation*}
\begin{split}
&\int_{\tilde\sigma^+((i-1)h)}^{\tilde\sigma^+ (t)} g(\tilde y(s), \tilde y(s-S))\tilde\alpha (s) ds 
\\
&\quad\qquad =\int_{\tilde\sigma^+((i-1)h)}^{\tilde\sigma^+ (t)} g(\tilde y(s), \tilde y(s-S))\tilde\alpha^0(s) ds
+ \underset{ t' \in ](i-1)h,t] }{\sum} (\tilde y (\tilde\sigma^+ (t')) - \tilde y(\tilde\sigma^- (t'))).
\end{split}
\end{equation*}
 In view of \eqref{mu_cont_2}, from Lemma \ref{lem_3_3} (applied to $F(\cdot)=g(\tilde y(\tilde\sigma(\cdot)), \tilde y(\tilde\sigma(\cdot)-S))\tilde\alpha^0(\tilde\sigma(\cdot)$,  $A=\tilde\sigma$ and $B=\tilde\psi$) 
 it follows that
 $$
\begin{array}{l}
\ds\int_{[(i-1)h,t]} g(x(t'), x(t'-h))d\mu^{c} (t')\\
\qquad\qquad\ds=\int_{[\tilde\psi(\tilde\sigma((i-1)h)),\tilde\psi(\tilde\sigma (t))]} g(\tilde y(\tilde\sigma(t')), \tilde y(\tilde\sigma(t')-S))\tilde\alpha^0(\tilde\sigma(t'))\,d\tilde\sigma(t') \\
\qquad\qquad\ds= \int_{\tilde\sigma^+((i-1)h)}^{\tilde\sigma^+ (t)} g(\tilde y(s), \tilde y(s-S))\tilde\alpha^0(s) ds.
 \end{array}
 $$
 Furthermore,  \eqref{tilde_y_and_zeta}, \eqref{tildeh_y_and_zeta} imply that 
 $$
 \underset{ t' \in ](i-1)h,t] }{\sum} (\tilde y (\tilde\sigma^+ (t')) - \tilde y(\tilde\sigma^- (t'))) 
=
 \underset{ r\ \in ]0, t-(i-1)h] }{\sum} (\zeta^{r}_i (1) - \zeta^{r}_i (0)) 
 $$
 if $i=1,\dots,N-1$ and $t\in](i-1)h,ih[$, or if $i=N$ and $t\in](N-1)h,Nh]$.
 
Notice that the function $x$ is right continuous on $]0,T[$, so that the boundary conditions on the $\zeta^r_i$'s can be expressed in terms of the function $x$:
%
%
\begin{equation*}
\zeta^r_i (0) =
\left\{
\begin{array}{ll}
x^-(r + (i-1)h) & \mbox{if } r \in ]0,h]
\\
\zeta_{i-1}^h (1)&\mbox{if } r =0
\end{array}
\right.\quad i=1,\ldots, N.
\end{equation*}
 Thus these $\zeta^r_i$'s  are consistent with those appearing in Def. \ref{def_ext}.  
%
Reviewing the above relations, we see that  $(x, \mu, \{w^{r}\}_{r \in [0,h]})$ is a process for (DIS). 
\vspace{0.05 in}
 
\noindent {\bf (B):} Take an extended process $(x,\mu, \{w^{r}\}_{ r \in [0,h]})$ for (DIS).   For any $i=1,\dots, N$,  let $\mu_i\in C^{\oplus}(0,h)$ be the Borel measure with distribution function 
$$
\mu_i([0,r]):=\begin{cases} \ds \int_0^1w_i^0(s)\,ds+\mu(](i-1)h,r+(i-1)h]) \quad\text{if $r\in[0,h[$,} \\
\ds \int_0^1w_i^0(s)\,ds+\mu(](i-1)h,h[)+\int_0^1w_i^h(s)\,ds \quad\text{if $r=h$.}
 \end{cases}
$$
Notice that $\mu_i(\{r\})=\mu(\{r+(i-1)h\})$ if $r\in]0,h[$, while $\mu_i(\{0\})=\int_0^1w_i^0(s)\,ds$ and $\mu_i(\{h\})=\int_0^1w_i^h(s)\,ds$, so that  Def. \ref{imp_control_def}\,(ii)\,,(iii) imply the relations
\bel{mu_i_rel}
\begin{array}{l}
\ds \int_0^1w_i^r(s)\,ds=\mu_i(\{r\}) \quad\text{for  $r\in[0,h]$, \ $i=1,\dots,N$,} \\[1.5ex]
\ds \mu_i(\{h\})+\mu_{i+1}(\{0\})=\mu(\{ih\}), \quad \mbox{for   }i=1,\dots,N-1.
\end{array}
\eeq
Define the strictly increasing function,  
\bel{phi}
\phi(0):=0, \quad \phi(r):= r + \sum_{i=1}^N \mu_i([0,r]) \quad\mbox{for   }r \in ]0,h ]\,,
\eeq
and let $S:= \phi(h)$.   The function  $\phi$ is right continuous on $]0,h[$. Take  $\eta:[0,S] \rightarrow [0,h]$ be the right inverse  of $\phi$, as defined in Lemma \ref{lem_3_3}.  The function $\eta$ is increasing,  1-Lipschitz continuous,   such that $\eta(0)=0$  and $\eta (S)=h$.  
 Write $\{r_j \}$ for the countable set of discontinuity points of   $\phi$ and define 
$
D := \underset{j}{\cup} [ \phi^-(r_j),\phi^+(r_j)]\,.
$ 
For each $r \in [0,h]$, $\eta^{-1}(r)=[ \phi^-(r),\phi^+(r)]$ and  $\eta$ is constant exactly on the intervals in $D$.  
Then, the Lebesgue measure $\mu_L$ and  the measures $\mu_i^{c}$, $i=1,\ldots,N$, are absolutely continuous w.r.t. the measure $d\phi$. Let $m_{L}$ and $m_i^c$ be the Radon-Nicodym derivatives of $\mu_L$ and  $\mu_i^{c}$, $i=1,\ldots,N$ w.r.t. $d\phi$, respectively. Hence,  $0\leq m_L(r) \leq 1$ and $0\leq m_i^c \leq 1$, $i =1,\ldots, N$, $d\phi$-a.e. Furthermore,
\begin{equation}
\label{normalize}
m_L(r) +\sum_{i=1}^N \, m_i^c (r) =1 , \quad\mbox{$d\phi$-a.e. } r \in [0,h]\backslash \underset{j}{\cup}\{r_j \}.
\end{equation}
Now  $0=\mu_L(\{r_j\})=m_L(r_j)[\phi^+(r_j)-\phi^+(r_j)]$ and $0= \mu_i^{c} (\{r_j\})= m_i^c (r_j)[\phi^+(r_j)-\phi^+(r_j)]$, so that 
$m_L(r_j)=0$, and $m_i^c (r_j) =0$  for all $j\in\N$.

We define
the Borel measurable functions $\alpha_0: [0,S] \rightarrow \R$ and $\alpha_i: [0,S] \rightarrow \R$, $i=1,\ldots,N$ to be 
\bel{alpha}
\alpha_0(s):= 
m_L (\eta( s)), \quad\text{for $s\in[0,S]$}
\eeq
(notice that $\alpha_0(s)=0$ on $D$) and
\bel{alpha_i}
\alpha_i(s):=
\left\{\begin{array}{ll}
   m_i^c (\eta( s)) & \mbox{if } s \notin D
\\
\frac{1}{\phi^+ (r_j)- \phi^- (r_j)}\, w^{r_j}_i \left(\frac{s- \phi^-(r_j)}{\phi^+ (r_j)- \phi^- (r_j)} \right)& \mbox{if }s\in 
[\phi^-(r_j), \phi^+(r_j)] \mbox{ for some } j\,.
\end{array}
\right.
\eeq
We   show that
\begin{equation}
\label{sum}
 \alpha_0(s) +  \sum_{i=1}^N \alpha_i (s) =1 , \quad\mbox{$\mu_L$-a.e. } s \in [0,S].
\end{equation} 
This is certainly true for $\mu_L$-a.e. $s \notin D$ by (\ref{normalize}). On the other hand, for any  $j$ and  $\mu_L$-a.e. $s \in [\phi^- (r_j), \phi^+ (r_j)]$, we know that $\alpha_0(s)= 0$ and 
\begin{equation*}
\begin{split}
\sum_{i =1}^N\alpha_i (s)&=\frac{1}{\phi^+ (r_j)- \phi^- (r_j)}\,\sum_{i =1}^N w^{r_j}_i \left(\frac{s- \phi^-(r_j)}{\phi^+ (r_j)- \phi^- (r_j)} \right)
= \frac{\sum_{i =1}^N \mu_i (r_j)}{\phi^+ (r_j)- \phi^- (r_j)}\,, 
\end{split}
\end{equation*}
by Def.  \ref{imp_control_def}\,(i), and \eqref{mu_i_rel}.  Since  
 $\phi^+ (r_j) - \phi^- (r_j) = \sum_{i =1}^N \mu_i (r_j)$, we obtain that \eqref{sum} 
 is true for $\mu_L$-a.e. $s \in [0,S]$.  
  Finally, using Lemma \ref{lem_3_3} for $B=\eta$,  $A=\phi$, and $F=m_L$, for any $s\in[0,S]$, we get
 \bel{important}
 \begin{split}
 \int_0^s\left(1- \sum_{i=1}^N \alpha_i (s')\right)\,ds'&= \int_0^s\alpha_0(s')\,ds'=\int_0^{\phi(\eta(s))}m_L (\eta( s))\,ds\\
 &=\int_{[0,\eta(s)]}m_L(r)d\phi(r) =\eta(s).
\end{split}
\eeq
Therefore, the functions $\eta$ and $\phi$ coincide with the function $\psi$ defined as in \eqref{psi} and its right inverse $\sigma$, respectively. Notice that the controls $\alpha_i\in L^1(0,S)$ are nonnegative,   $\sum_{i=1}^{N} \alpha_i(s) \leq 1$  a.e.  $s\in [0,S]$, and  $\int_0^S  (1- \sum_{i=1}^{N} \alpha_i (s))ds =\eta(S)=h$.
  We next construct functions 
 $y_i :[0, S] \rightarrow \R^n$  from the process 
$(x, \mu, \{w^{r}\}_{ r \in [0,h]})$ and the corresponding functions $\zeta^r_i :[0,1] \rightarrow \R^n$  appearing in Def. \ref{def_ext}. Then, we show that $(S,\{y_i\},\{\alpha_i\})$ is the (unique) reparameterized process associated with $S>0$ and $\{\alpha_i\}$. 
%
%
%
%
 Set $y_0(s):=\xi_0$ if $s\in[0,S[$ and $y_0(S):=x_0$. For $i=1,\ldots,N$, define
$$
y_i(0):= \begin{cases} x_0 &\quad\text{if $i=1$}\\
\zeta^h_{i-1} (1) &\quad\text{if $i>1$}
\end{cases}, \qquad y_i(S):= \zeta^h_i (1)
$$
$$
y_i(s):= \begin{cases}  x((i-1)h+\eta(s)) &\quad\text{if}\quad  s\in]0,S[\setminus  D\\
  \zeta^{r_j}_i \left(\frac{s- \phi^-(r_j)}{\phi^+ (r_j)- \phi^- (r_j)} \right) &\quad\mbox{if }s\in ]0,S[\cap 
[\phi^-(r_j), \phi^+(r_j)] \mbox{ for some  $j$.}
\end{cases}
$$
We deduce from the differential equations and boundary conditions for the $\zeta^r_i$'s that, for any index value $j$  of the discontinuity points $\{r_j\}$ of $\phi$, the $y_i$'s satisfy 
\begin{equation}
\label{jump_ode3}
\frac{dy_i}{ds}  (s)  \,=\,\alpha_i (s) g(y_i (s), y_{i-1} (s) ), \mbox{ a.e. } s \in [\phi^- (r_j), \phi^+ (r_j)] \mbox{ for } i =1,\ldots, N\,,
\end{equation}
with boundary conditions
\begin{equation}
\label{jump_ode4}
y_i  (\phi^- (r_j)) = 
\zeta^r_i (0) \qquad\text{($\zeta^r_i (0)$ as in \eqref{jump_ode2}).}
\end{equation}
 For  $i =1,\ldots,N$, take any 
$s \in ]0,S[$. 
Let $t =(i-1)h+\eta(s)$. 
Then, 
\begin{equation*}
\begin{split}
x(t) =   \zeta^0_i (1) &+\int_{(i-1)h}^t  f(x(t'), x(t'-h))dt'  + \int_{[(i-1)h,t]} g(x(t'), x(t'-h))d\mu^{c}(t')
\\
&+ \underset{ \{j\,:\, r_j\in]0,t-(i-1)h] \}}{\sum}  y_i (\phi^+(r_j) )- y_i (\phi^- (r_j) ).
\end{split}
\end{equation*}
(We have made use of properties of $y_i$ on $D$ to obtain  the last term on the right.)  But, in consequence of Lemma \ref{lem_3_3}, 
\begin{equation*}
 \begin{split}
 &\int_{(i-1)h}^t  f(x(t'), x(t'-h))dt' \\
 & \ \ \ =\int_{0}^{t-(i-1)h } f(x(r'+(i-1)h), x(r'+(i-1)h-h))dr' \\
&\ \ \ = \int_{[0,t-(i-1)h]} f(x(r'+(i-1)h), x(r'+(i-1)h-h))m_L (r') d\phi(r')
\\
&\ \ \  =\int_{[0,\phi^+ (t-(i-1)h)]}   f(x(\eta (s')+(i-1)h), x  (\eta (s')+(i-1)h-h)) m_L(\eta (s'))\,ds'
\\
&\ \ \  =\int_{[0, s]}  f(y_i(s'), y_{i-1} (s'))\alpha_0(s')\,ds'\,.
\end{split}
\end{equation*}
(We have used the facts that $y_i (s')=x(\eta(s')+(i-1)h)$ for $m_L$-a.e. $s' \notin  D$, and  
 $\alpha_0(s)=0$,  a.e. on  $[\phi^- (t-(i-1)h), \phi^+ (t-(i-1)h)]$.)

\vspace{0.05 in}
\noindent Furthermore (since   $m_i^c (\eta (s))=0$ a.e. on  $[\phi^- (r), \phi^+ (r)]$, 
so that $s'=\phi(\eta(s'))$ if $s'\notin D$), applying Lemma  \ref{lem_3_3} we get
\begin{equation*}
 \begin{split}
&\int_{[(i-1)h,t]} g(x(t'), x(t'-h))d\mu^{c}(t')\\
&\ \ \ = \int_{[0,t-(i-1)h]} g(x(r'+(i-1)h), x(r'+(i-1)h-h))d\mu_i^{c}(r') \\
&\ \ \   = \int_{[0,t-(i-1)h]}  g(x(r'+(i-1)h), x(r'+(i-1)h-h))m_i^c  (r') d\phi(r')
\\
&\ \ \  =\int_{[0,\phi^+ (t-(i-1)h)]}  g(x(\eta (s')+(i-1)h), x  (\eta (s')+(i-1)h-h) ) m_i^c (\eta (s'))\,ds'
\\
&\ \ \  =\int_{[0, s]}  g(y_i(s'), y_{i-1} (s'))\alpha_i(s')   \chi_{[0,S]\backslash D}(s')ds'\,.
\end{split}
\end{equation*}
If $s\in]0,S[\setminus  D$,  then   we conclude from the above relations that
\begin{equation*}
\begin{split}
y_i(s) = \zeta^0_i (1)   &+ \int_{[0, s]}  f(y_i(s'), y_{i-1} (s'))\alpha_0 (s')\,  ds' \\
&+  \int_{[0, s]} g (y_i(s'), y_{i-1} (s'))\alpha_i(s')   \chi_{[0,S]\backslash D}(s')ds'\\
 &+ \underset{ \{j\,:\, r_j\in]0,t-(i-1)h] \}}{\sum}  y_i (\phi^+(r_j) )- y_i (\phi^- (r_j) ).
 \end{split}
 \end{equation*}
But $ \zeta^0_i (1)=y_i(0)+ (y_i (\phi^+(0) )- y_i (\phi^- (0) ))$, so
\begin{equation*}
\begin{split}
 \zeta^0_i (1)&+ \underset{ \{j\,:\, r_j\in]0,r-(i-1)h] \}}{\sum}  y_i (\phi^+(r_j) )- y_i (\phi^- (r_j) ) \\
& = y_i(0)+ \underset{ \{j\,:\, r_j\in[0,t-(i-1)h] \}}{\sum}  y_i (\phi^+(r_j) )- y_i (\phi^- (r_j) ) \\
&=y_i(0)+   \int_{[0, s]}  g(y_i(s'), y_{i-1} (s'))\alpha_i(s')   \chi_{D}(s')ds'\,.
\end{split}
\end{equation*}
It follows that, for each $i$, 
\begin{equation}
\label{y_eqn}
y_i(s) =y_i (0)  + \int_{[0, s]}  f(y_i(s'), y_{i-1} (s'))\alpha_0(s') ds' +
  \int_{[0, s]}  g(y_i(s'), y_{i-1} (s'))\alpha_i(s') ds' .
 \end{equation}
 Now assume that   $s \in ]0,S[\cap[ \phi^- (r_j)  , \phi^+ (r_j) ]$ for some index value $j$. Since $\phi^- (r_j)$ is a density point of $[0,S]\backslash D$ and the $y_i$'s are continuous, we can deduce from  (\ref{y_eqn}) that, for each $i$,
\begin{equation*}
\begin{split}
y_i(\phi^- (r_j)) =y_i (0)  &+ \int_{[0, \phi^- (r_j)]}  f(y_i(s'), y_{i-1} (s'))\alpha_i(s')  ds'
\\
& +
  \int_{[0,\phi^- (r_j)]}  g(y_i(s'), y_{i-1} (s'))\alpha_i(s')\, ds' .
\end{split}
\end{equation*}
 Exploiting the facts that
\bel{jump_y}
y_i( s) -y_i(\phi^- (r_j)) = \int_{\phi^- (r_j)}^s \alpha_i (s')g(y_i (s'), y_{i-1}(s'))ds'
\eeq
and $\alpha_0(s) =0$ for $\mu_L$-a.e. $s \in [ \phi^- (r_j)  , \phi^+ (r_j) ] $, we see that \eqref{y_eqn} is valid for arbitrary $s \in ]0,S[$\,.  By definition, $y_i(S)= \zeta_i^h(1)$. Hence, if $S\notin D$ or $i=N$,  then $y_i $ satisfies \eqref{y_eqn} on $]0,S]$ by continuity.  If, on the other hand,  $i\in\{1,\dots,N-1\}$ and $S \in [\phi^- (r_j), \phi^+ (r_j) ]$  for some index value $j$, then $r_j =h$ and we derive that \eqref{y_eqn} still holds on  $]0,S]$ by \eqref{jump_y}.   
Furthermore, notice that the  $y_i$'s satisfy the boundary conditions in Def.  \ref{def_rep}. In particular,  
$
y_1(0)=x_0$ and  $y_{i+1}(0)=\zeta_i^h(1)=y_i(S).
$
Reviewing the preceding relations, we see that assertion (B)  has been confirmed.
\end{proof}
 The mapping, constructed in the proof of Part (B) of the theorem, between extended and reparameterized processes, is actually invertible.
 
\begin{corollary}\label{Cor_inv} Let  ${\mathcal I}$ be the map which associates with any extended process $(x,\mu,  \{w^{r}\}_{ r \in [0,h]})$   the reparameterized process  $(S , \{y_i\}, \{\alpha_i\})$ for {\rm (DIS)} constructed  in the proof of Thm. \ref{Prop_3_2},\,{\rm (B)}.  Then, the function ${\mathcal I}$ is invertible and  the extended process $(x,\mu,  \{w^{r}\}_{ r \in [0,h]})={\mathcal I}^{-1}(S , \{y_i\}, \{\alpha_i\})$ coincides with the extended process associated with $(S , \{y_i\}, \{\alpha_i\})$ constructed in  the proof of Thm. \ref{Prop_3_2},\,{\rm (A)}.
\end{corollary}
\begin{proof} This assertion is an easy consequence of relation \eqref{important}, which implies that the change of variable $r\mapsto \phi(r)$ (see \eqref{phi}) utilized in the proof of Thm. \ref{Prop_3_2}, statement (B), to obtain $(S, \{y_i\}, \{\alpha_i\})$ from $(x,\mu, \{w^{r}\}_{ r \in [0,h]})$ is the right inverse of the change of variable $s\mapsto \psi(s)$ (see \eqref{psi}) employed in the proof of part {\rm (A)} to construct  an extended process
from $(S, \{y_i\}, \{\alpha_i\})$. (See also the arguments in the proof of Prop. \ref{P_uniq} below).
\end{proof}

\section{Existence, Density and Compactness Properties of Extended Trajectories}  \label{sec4}
We establish some fundamental properties of extended processes.

\begin{proposition}\label{P_uniq} Given an impulsive control $(\mu, \{w^{r}\}_{ r \in [0,h]})$, there exists an extended trajectory $x\in BV([-h,T];\R^n)$ to {\rm (DIS)}  associated with it and it is unique.
\end{proposition} 
 \begin{proof}  Take an impulsive control $(\mu, \{w^{r}\}_{ r \in [0,h]})$. Define $S>0$ and the controls  $\{\alpha_i\}$ as in the proof of assertion (B) in Thm. \ref{Prop_3_2},  and let $\{y_i\}$ be the corresponding (unique) solution  to the reparameterized system. Then,  with the reparameterized process $(S, \{y_i\}, \{\alpha_i\})$ we can associate an extend process $(\hat x, \hat\mu, \{\hat w^{r}\}_{ r \in [0,h]})$  as in  the proof of assertion (A). Since  $\sigma\equiv\phi$ (we use the notation of the proof), it is not difficult to deduce  that the impulsive control $(\hat\mu, \{\hat w^{r}\}_{ r \in [0,h]})$ actually coincides with  $(\mu, \{w^{r}\}_{ r \in [0,h]})$. Therefore,   $\hat x$ is an extended trajectory associated with the given impulsive control.  Suppose now that $x$ is another $BV$ function such that    $(x,\mu, \{w^{r}\}_{ r \in [0,h]})$ is an extended process. To this process there corresponds the same reparameterized process $(S, \{y_i\},  \{\alpha_i\})$ as above,  since, given $S>0$ and  $\{\alpha_i\}$,   the solution  $\{y_i\}$  is unique. The equality $x=\hat x$ now follows from the  properties of the $y_i$'s constructed in  the proof of  (B). 
 \end{proof}

For any $K>0$, we define the {\em $K$-extended reachable set for {\rm (DIS)},}  as
\begin{equation*}
{\mathcal R}^{{\rm e}}_K(T):= \{ x(T):\,   (x, \mu,  \{w^{r}\}_{ r \in [0,h]}) \mbox{ is extended}\mbox{ process for }{\rm (DIS)}  \mbox{ s.t. }  ||\mu ||_{\mbox{TV}} \leq K\}
\end{equation*}
and the  {\em $K$-strict sense reachable set  for {\rm (DIS)}}, as 
\begin{equation*}
{\mathcal R}^{{\rm s}}_K(T):= \{ x(T):\,   (x, u) \mbox{ is a strict sense }\\
 \mbox{process for }{\rm (DIS)}\mbox{ s.t. } \int_0^Tu(t)\,dt \leq K\}.
\end{equation*}
Clearly, ${\mathcal R}^{{\rm s}}_K(T)$ is  nonempty and  ${\mathcal R}^{{\rm s}}_K(T)\subseteq{\mathcal R}^{{\rm e}}_K(T)$. In control theory an extension of a family of processes is said to be a `relaxation' of that family, if the reachable set of the original family is dense in the reachable set of the extension and if the extended reachable is closed. The role of the extension is then to be `close' to the original family of processes, but to include extra elements ensuring closure properties. Relaxation procedures have had an important role in optimal control and system theory, both in the investigation of conditions that guarantee existence of solutions to optimal control problems and in the derivation of necessary conditions of optimality.  Proposition \ref{Pext} below justifies interpreting extended sense processes for (DIS) as a relaxation of the family of strict sense processes. 
\begin{proposition}\label{Pext}  For any $K>0$,    ${\mathcal R}^{{\rm e}}_K(T)$ is a non-empty compact set  and  ${\mathcal R}^{{\rm e}}_K(T)= \overline{{\mathcal R}^{{\rm s}}_K(T)}$.
\end{proposition}
\begin{proof} Let us first show that $ {\mathcal R}^{{\rm e}}_K(T)\subseteq \overline{{\mathcal R}^{{\rm s}}_K(T)}$. 
Take $z\in {\mathcal R}^{{\rm e}}_K(T)$ and let us consider an extended process $(x, \mu, \{w_i^{r}\}_{ r \in [0,h]})$ with $||\mu ||_{\mbox{TV}} \leq K$ and such that $x(T)=z$. Then, by Thm. \ref{Prop_3_2}, there exists   $(S, \{y_i\}, \{\alpha_i\})$ which is a reparameterized process for {\rm (DIS)} such that $x(T)= y_N (S)$  and $\sum_{i=1}^N\int_0^S\alpha_i(s)\,ds=||\mu ||_{\mbox{TV}}$. Set
$
\alpha_0(s):=1-\sum_{i=1}^N \alpha_i(s) 
$
for $s\in[0,S]$, and $\nu:=\sum_{i=1}^N\int_0^S\alpha_i(s)\,ds$.
For each index value $j\in\N$, $j\ge1$, and a.e. $s\in[0,S]$,    define
$
\alpha_0^j(s):=\frac{h}{h+\frac{\nu}{j}}\left[\alpha_0(s)+\frac{1}{j}\sum_{i=1}^N\alpha_i(s)\right],
$
and 
$$
\alpha^j_i (s) := 
\left\{
\begin{array}{ll}
\ds\left(1-\alpha_0^j(s)\right)\frac{\alpha_j(s)}{\sum_{i=1}^N\alpha_i(s)} & \mbox{if } \sum_{i=1}^N\alpha_i(s)\ne0
\\
\ds\frac{\left(1-\alpha_0^j(s)\right)}{N} &\mbox{if } \sum_{i=1}^N\alpha_i(s)=0
\end{array}
\right., \qquad i=1,\ldots, N.
$$
Notice that $\alpha_0^j(s)+ \sum_{i=1}^N\alpha^j_i(s)=1$  and  $0<\frac{1}{j}\le \alpha_0^j(s)\le1$ a.e. Moreover,  since $S=h+\nu$, 
one has $\int_0^S\left(1- \sum_{i=1}^N\alpha^j_i(s)\right)\,ds=\int_0^S  \alpha_0^j(s)\,ds=h$ and 
$$
\int_0^S\sum_{i=1}^N\alpha^j_i(s)\,ds=\int_0^S  \left(1-\alpha_0^j(s)\right)\,ds=S-h=\nu\le K.
$$
Let $(y_1^j, \dots,y_N^j)$ be the unique solution to the reparameterized system associated with $S$ and the $\alpha_i^j$'s. Then, for any $j\ge1$, the element $(S, \{y_i^j\}, \{\alpha_i^j\})$ is a reparameterized process for {\rm (DIS)} to which by Thm.  \ref{Prop_3_2}  it corresponds a strict sense process $(x^j,u^j)$  such that $x^j(T)=y^j_N(S)$ and $\int_0^Tu^j(s)\,ds=\nu\le K$. 
 We see that, for any $i=1,\dots,N$, the controls $\alpha_i^j$ converge to $\alpha_i$ in the $L^\infty$-norm, so that, by the continuity of the input-output map, 
one has
$ |x^j (T) - x(T)|=|y^j_N (S) - y_N (S)|\to 0$.
Therefore, we have shown that all points in ${\mathcal R}^{{\rm e}}_K(T)$ lie in the closure of ${\mathcal R}^{{\rm s}}_K(T)$. 

\vsm
Let us now prove that  ${\mathcal R}^{{\rm e}}_K(T)$ is compact. Take a  sequence $(z_j)_j\subset {\mathcal R}^{{\rm e}}_K(T)$. For each $j$ then, there exists an extended process $(x^j,\mu^j, \{w^{j, r}\}_{ r \in [0,h]})$ such that $x^j(T)= z_j$ and  $||\mu_j ||_{TV} \leq K$.  By Thm. \ref{Prop_3_2} there exists a reparameterized process  $(S^j, \{y^j_i\}, \{\alpha^j_i\})$ for {\rm (DIS)} such that $y^j_N (S^j)= z_j$. Furthermore, it holds  $||\mu^j ||_{TV} = \underset{i}{\sum}\; \int_0^{S^j} \alpha^j_i (s)ds$. But then one obtains
$
h = \int_{[0,S^j]} (1- \sum_i \alpha^j_i (s)) ds = S^j - ||\mu^j||_{TV}, 
$
so that $S^j \leq h +K$.  Since the sequence  $(S^j)_j$ is bounded, possibly passing  to a subsequence (we do not relabel) we have that $S^j$ converges to some  $S\le h+K$.  Let us now  add to the reparameterized system the equations of time and total variation, denoted by $\tau$ and $v$, respectively, that is we consider the system
\bel{sys_tv}
\ \left\{
\begin{array}{l}
\ds\dot\tau(s)=\Big(1- \sum_{i=1}^{N} \alpha_i (s)\Big) \\
\ds\dot v(s)= \sum_{i=1}^{N} \alpha_i (s)  \\
\ds\dot y_i (s) =\Big(1- \sum_{i=1}^{N} \alpha_i (s)\Big) f(y_i (s),y_{i-1}(s)) + \alpha_i (s)g( y_i(s), y_{i-1}(s)),    \ \
\mbox{$i=1,\ldots, N$}, 
\end{array}
\right. 
\eeq
with the previous boundary conditions on   $\{y_i\}$ and $\tau(0)=v(0)=0$. For each $j\ge 1$, let $(\tau^j, v^j,\{y_i^j\})$ denote the  solution to \eqref{sys_tv}   on $[0,S^j]$, associated with   the  controls $\{\alpha_i^j\}$ and satisfying these endpoint conditions. Notice that  $\tau^j(S^j)=h$ and  $v^j(S^j)\le K$.
We may apply a standard convergence analysis, based on the Compactness of Trajectories Theorem \cite[Thm. 2.5.3]{OptV} and the fact that the velocity sets associated with \eqref{sys_tv} are convex, to show that there exists a solution $(\tau,v, \{y_i\})$ to \eqref{sys_tv} on $[0,S]$, corresponding to some control $\{\alpha_i\}$, 
 such that, along a subsequence (we do not re-label), $(\tau^j, v^j,\{y_i^j\})$ converges uniformly to  $(\tau,v, \{y_i\})$ on $[0,S]$ (possibly extending by a constant, continuous extrapolation all functions to $[0,h+K]$). In particular, $(S, \{y_i\},\{\alpha_i\})$ is a reparameterized process for (DIS) with $\sum_{i=1}^N\int_0^S\alpha_i(s)\,ds\le K$, since 
 $
\tau(S)=\lim_j\tau^j(S^j)=h
$  and $v(S)=\lim_j v^j(S^j)\le K$.
Moreover, $\lim_j z^j=\lim_j  y^j_N (S^j) = y_N (S)$.  Again,  appealing to Thm. \ref{Prop_3_2}, we deduce that there exists an extended process
$(x, \mu, \{w^{r}\}_{ r \in [0,h]})$ such that $x(T)=y_N (S)$ and $\|\mu\|_{TV}\le K$.  This proves that  ${\mathcal R}^{{\rm e}}_K(T)$ is compact and ${\mathcal R}^{{\rm e}}_K(T)\subseteq\overline{{\mathcal R}^{{\rm s}}_K(T)}$. Since ${\mathcal R}^{{\rm s}}_K(T)\subseteq{\mathcal R}^{{\rm e}}_K(T)$,  the compactness of ${\mathcal R}^{{\rm e}}_K(T)$ yields ${\mathcal R}^{{\rm e}}_K(T)=\overline{{\mathcal R}^{{\rm s}}_K(T)}$.
%
\end{proof}

\section{Existence of Optimal  Extended Processes and a Maximum Principle}\label{S_PMP}
We have defined extended processes for the delayed  impulsive control system {\rm (DIS)}. We now consider a related optimal control problem over extended processes for {\rm (DIS)} with terminal cost, involving constraints on the total variation of the measure control and the location of the terminal state. 
\vsm
 
\noindent
$
{\rm (P)} \left\{
\begin{array}{l}
\mbox{Minimize } \Psi(x(T))
\\
\mbox{over extended processes $(x,\mu,  \{w^{r}\}_{ r \in [0,h]})$ for {\rm (DIS)}, satisfying}
\\[1.0ex]
\|\mu\|_{TV} \leq K,
\\[1.5ex]
x(T) \in \T \,.
\end{array}
\right.
$
\vspace{0.05 in}

\noindent
The additional data for this problem comprise a  function $\Psi:\R^n \rightarrow \R$, a  nonempty set $\T \subset \R^n$,  and $K >0$. 
\vsm
\noindent We shall say that an extended process $(x,\mu, \{w^{r}\}_{ r \in [0,h]})$ for {\rm (DIS)}  is {\it admissible} if $x(T) \in \T$ and $\|\mu\|_{TV} \leq K$. An admissible extended process  $(\bar x,\bar \mu,  \{\bar w^{r}\}_{r \in [0,h]})$ for {\rm (DIS)} that minimizes $\Psi(x(T))$ over all admissible extended processes \linebreak $(x,\mu, \{w^r\}_{ r \in [0,h]})$ for {\rm (DIS)} is said to be an {\it optimal extended process}.

\begin{theorem}[Existence of Optimal Controls]\label{Exist}
Consider the optimal control problem {\rm (P)}. Assume  that $\T$ is closed and  $\Psi:\R^n\to\R$ is lower semicontinuous   and that there exists an admissible extended process. Then, there exists an optimal extended process.
\end{theorem}

\begin{proof}   Under the stated hypotheses, by Prop. \ref{Pext} the set ${\mathcal R}^{{\rm e}}_K(T)\cap \T$ is non-empty and compact. Since the function $\Psi$ is lower semicontinuous, there exists $\bar z \in {\mathcal R}^{{\rm e}}_K(T)\cap \T$ that minimizes $\Psi$ over this set. Let $(\bar x, \bar \mu,  \{\bar w^{r} \}_{ r \in [0,h]})$ be an admissible extended process  such that $\bar x(T) = \bar z$. Take any admissible extended process $(x, \mu,  \{w^{r}\}_{r \in [0,h]})$ and write $z= x(T)$. Then, by the minimizing property of $\bar z$,
$$
\Psi(\bar x(T))= \Psi(\bar z)\leq \Psi(z)= \Psi(x(T)).
$$
In other words $(\bar x, \bar \mu, \{\bar w^{r}\}_{ r \in [0,h]})$ is a minimizer for (P).
\end{proof}


The maximum principle of this section is derived by transforming the impulsive optimal control problem into a conventional, free end-time optimal control problem, to which we apply standard necessary conditions of optimality. In the following, given a reference extended process $(\bar x,\bar \mu,  \{\bar w^{r}\}_{ r \in [0,h]})$, for any function $h=h(x_1,x_2)$, $(x_1,x_2)\in  \R^n\times\R^n$,  we set  $ \bar h (t):= h(\bar x (t), \bar x (t-h))$ and write $\bar h_{x_1}(t)$, $\bar h_{x_2}(t)$ for the Jacobian of $h$ at $(\bar x (t), \bar x (t-h))$  w.r.t. $x_1$  and $x_2$, respectively. For instance, $ \bar f (t)= f(\bar x (t), \bar x (t-h))$ and $ \bar f_{x_1} (t)= \frac{\partial f}{\partial x_1}(\bar x (t), \bar x (t-h))$. 

\begin{theorem}[Maximum Principle] \label{PMP}
Let  $(\bar x,\bar \mu,  \{\bar w^{r}\}_{ r \in [0,h]})$  be a minimizing process for {\rm (P)}.   For $i  = 1,\ldots,N$ and $r \in [0,h]$, let $\bar\zeta^r_i:[0,1]\rightarrow \R^n$ be the corresponding functions describing the instantaneous evolution of the state of this process, defined by \eqref{jump_ode1} and \eqref{jump_ode2}. Assume that $\T\subseteq\R^n$ is closed and $\Psi:\R^n\to\R$ is a $C^1$ function.  
Then,  there exist $\lambda \geq 0$, $c \in \R$, $d \geq0$ and a function $p\in BV([0,+\infty[; \R^n)$, right continuous on $]0,T[$, with the following properties: $d = 0$ if $\|\bar \mu\|_{TV} < K$ and
\vspace{0.05 in}

\noindent
 (A): $\lambda +\|p\|_{L^\infty}\neq 0$,
\vspace{0.05 in}
 
\noindent
(B):
 for each $i =1,\ldots,N$ and $t \in [(i-1)h, ih]$,
\begin{equation*}
\begin{split}
&p(t)= 
 \eta_i^0 (1), \mbox{ when  $i =2,\dots,N$ and  $t= (i-1)h$,} 
\\
& p(t)= \eta_i^0 (1) - \int_{(i-1)h}^t p(t') \cdot \bar f_{x_1}(t')dt'  - \int_{[(i-1)h , t]} p(t') \cdot \bar g_{x_1}(t') \bar d\mu^{c}(t')  
 \\
 &\ \ - \int_{(i-1)h}^t  p(t'+ h) \cdot \bar f_{x_2}(t'+ h)dt'  - \int_{[(i-1)h , t]}  p(t'+h) \cdot \bar g_{x_2}(t'+h))   \bar d\mu^{c}(t'+h) 
\\
&\ \ - \underset{r \in ]0 , t- (i-1)h] }{\sum} ( \eta^r_i ( 0) - \eta^r_i (1)),
 \mbox{ when  $t\in ](i-1)h,ih[$,}
\end{split}
\end{equation*}
\vspace{-0.1 in}

\noindent
Furthermore,
$$
p(T) = \eta^h_N (1) \mbox{ and }
p(t)= 0 \mbox{ for }t >T. 
$$
Here, for each $r \in [0,h]$,
$(\eta^r_1,\ldots, \eta^r_N) :[0,1] \rightarrow (\R^{n})^N$ satisfies the differential equations

\begin{eqnarray} \label{p_rho}
&&
 \frac{d}{ds} \eta^r_i (s)  \,=\,  -\eta^r_i (s) \cdot g_{x_{1}} (\bar\zeta_i^r(s),\bar\zeta_{i-1}^r (s))\bar w_i^{r}(s)\,
\\
&&\qquad-  \; \eta^r_{i+1} (s) \cdot  g_{x_{2}}(\bar\zeta_{i+1}^r(s),\bar\zeta_{i}^r(s) )\bar w_{i+1}^{r}(s)
\qquad \mbox{ a.e. } s \in [0,1],\ \ i =1,\ldots, N\,, \notag
\end{eqnarray}
with boundary conditions
\begin{equation}
\label{jump_ode2_1}
\eta^r_i (0) =
\left\{
\begin{array}{ll}
p^-(r + (i-1)h) & \mbox{if } r \in ]0,h]
\\
\eta_{i-1}^h (1)&\mbox{if } r =0 \,.
\end{array}
\right.
\end{equation}
In these relations, 
we interpret
$\eta^h_{0} (1) := p(0)$ and $\eta^r_{N+1} (s):= 0$ for all $r \in [0,h]$ and $s \in [0,1]$.

\vspace{0.05 in}

\noindent
 \vspace{0.05 in}

\noindent
 (C): $-p(T) \in \lambda \, \nabla\Psi (\bar x(T)) + N_\T (\bar x(T))$, 
 \vspace{0.05 in}

\noindent
 (D): 
 
 \begin{itemize}
%
 %
%
\item[(i):]
$\underset{j=1}{\overset{N}{\sum}} \,  p(r+(j-1)h)\cdot \bar f(r+ (j-1)h ) - c =  0$ \  for all  \ $r \in ]0,h[$, 

\noindent and also at $r=0$ if $\underset{j=1}{\overset{N}{\sum}}\int_0^1\bar w^0_j(s)\,ds=0$, and   at  $r=h$ if  $\underset{j=1}{\overset{N}{\sum}}\int_0^1\bar w^0_j(s)\,ds=0$ and   $\underset{j=1}{\overset{N}{\sum}}\int_0^1\bar w^h_{j} (s)\,ds=0$, 

\vspace{0.05 in}

  \item[(ii):]
 $p(t)\cdot \bar g(t)-d  \le 0$
 for all $t \in[0,T]$, 
\vspace{0.05 in}

\item[(iii):]   $\mbox{supp}\, \{{\bar \mu} \}\subset\left\{t\in[0,T]: \ \ p(t)\cdot \bar g(t)-d =0\right\}$,
%
 
 \vspace{0.05 in}
 
  \item[(iv):]   for any $r\in[0,h]$  such that $\underset{j=1}{\overset{N}{\sum}}\int_0^1\bar w_j^r(s)\,ds >0$,    
 \begin{itemize}
 \item[(a):] 
 $\mbox{$\eta^r_i (s) \cdot g( \bar\zeta^{r}_i (s),  \bar\zeta^{r}_{i-1} (s))  = \underset{j}{\max}\; \eta^r_j (s) \cdot g( \bar\zeta^{r}_j (s), \bar\zeta^{r}_{j-1} (s))$ }$,

\hspace{1.0 in}  for $\mathcal L$-a.e. $s\in [0,1]$ such that  $\bar w^r_i (s) >0 $. 
 
 \item[(b):]  $\underset{j=1}{\overset{N}{\sum}}  \eta^r_j (s)\cdot  f(\bar\zeta_j^r (s),\bar\zeta_{j-1}^r(s)))  -  c \leq \max_j  \eta_j^r(s)\cdot g(\bar\zeta_j^r (s),\bar\zeta_{j-1}^r(s))  -d =0$
 
 \hspace{3.0 in}  for all $s\in [0,1]$.
 \end{itemize}
 \end{itemize}
%
%

\end{theorem}

\noindent
{\bf Comments.}
\vspace{0.05 in}

\noindent
{\bf (1):} Condition (B) is the costate equation in integral form for $p$. Its corresponding infinitesimal form, obtained by differentiating across the relevant  equations at points of differentiability is a generalized `advance' differential equation (equivalently, a `delay' differential equation in reverse time), consistent with standard first order necessary optimality conditions for impulsive-free optimal control problems. Notice that  the `attached' controls $s \mapsto  \bar w^r_i (s)$
affect the jumps in the costate function $p$.  In particular, let us observe that $p$ may jump at some $t\in[(i-1)h,ih]$ where $\bar\mu$ (and $\bar x$) does not, as the ODE \eqref{p_rho} involves both the attached controls $\bar w_i^{t-(i-1)h}$ and $\bar w_{i+1}^{t-(i-1)h}$.
\vspace{0.05 in}

\noindent
{\bf (2):}  Condition (D),\,(i)  and   (D),\,(iii) are a version of the `constancy of the Hamiltonian condition' appropriate for the impulsive control problem.  Specifically, the equality in condition (D),\,(i) is satisfied also at $r=0$ if $t=0$ is not an atom for $\bar\mu$ and at  $r=h$ if {\em all} points $t=ih$, $h=1,\dots,N$ are not atoms for $\bar\mu$. 

\vspace{0.05 in}

\noindent
{\bf (3):}  Condition (D),\,(iv), which is related to the Weierstrass condition of classical optimal control, combines with (D),\,(iii) to give  information about the location of the support of the measure  $\bar\mu$.  In particular, condition  (D),\,(iv)  describes properties of the attached controls determining the instantaneous evolution of the state at atoms of $\bar\mu$. 

\vspace{0.05 in}

\noindent
{\bf (4):}  In view of the nontriviality relation (A), the transversality condition (C) and the linearity of the adjoint equations in (B), if the final point $x(T)$ is unconstrained we can choose $\lambda =1$.

\vspace{0.05 in}

\noindent
\begin{proof}[Proof of Thm. \ref{PMP}] Let   $(\bar S, \{\bar y_i\}, \{\bar\alpha_i\})$ be the reparameterized process associated with the optimal extended process  $(\bar x, \bar\mu, \{\bar w^{r}\}_{r \in [0,h]})$, constructed as in the proof of  Thm.  \ref{Prop_3_2}. Define 
$\bar \tau(s):=\int_0^s\Big(1-\underset{j=1}{\overset{N}{\sum}}\bar\alpha_j(s')\Big)\,ds'$ and $\bar v(s):=\int_0^s\underset{j=1}{\overset{N}{\sum}}\bar\alpha_j(s'))\,ds'$ for $s\in[0,\bar S]$,
and
$$
\Lambda^{(N)}:=\Big\{(a_1,\dots,a_N)\in\R^N: \quad a_j\ge 0 \text{ for each $j=1,\dots,N$, and } \sum_{j=1}^N a_j\le 1\Big\}.
$$
  It can be deduced from  Thm.  \ref{Prop_3_2} that  $(\bar S, \bar\tau,\bar v,\{\bar y_i\}, \{\bar\alpha_i \})$  is a minimizer for the following  optimal control problem: 
 $$
\left\{\begin{array}{l}
\mbox{Minimize } \Psi(y_N (S))
\\
\mbox{over controls $S>0$, $\alpha_i \in L^\infty (0,S)$,  and functions }\\
(\tau_i,v_i,y_i) \in W^{1,1}([0,S];\R^{1+1+n}),  \text{ for $i=1,\ldots, N$,    that satisfy} 
\\
\ds\dot\tau(s)=1-\sum_{j=1}^N \alpha_j(s) \qquad \text{a.e. $s\in[0,S]$,}
\\
\dot v(s)=\sum_{j=1}^N \alpha_j(s) \qquad \text{a.e. $s\in[0,S]$,}\\
\dot y_i (s) =(1- \sum_{j=1}^{N} \alpha_j (s)) f(y_i (s),y_{i-1}(s)) + \alpha_i (s)g( y_i(s), y_{i-1}(s))
\\ 
\hspace{2.5 in} \mbox{a.e. $ s \in [0,S]$, for $i=1,\ldots, N$}, \\ 
(\alpha_1(s),\dots,\alpha_N(s))\in \Lambda^{(N)} \quad \mbox{a.e. } s\in [0,S] ,
\\
y_1 (0)=x_0 \mbox{ and }  y_{i+1}(0) =y_{i}(S), \mbox{ for } i=1,\ldots , N-1, \quad y_{0}\equiv \xi_0,
\\
\tau(0)=0, \quad v(0)=0, \quad  \tau(S)=h, \quad  v(S)\leq K, \quad
y_N(S) \in \T.
\end{array}
\right.
$$
This is a standard `free end-time' optimal control problem for which necessary conditions are well known (see, e.g. \cite[Thm. 8.7.1]{OptV}). We deduce existence of $\lambda\geq 0$, $c\in \R$,  $d \geq 0$ ($-c$ and $-d$ is the Lagrange multiplier associated with $\bar\tau$ and $\bar v$, respectively) and absolutely continuous functions $q_i: [0,\bar S] \rightarrow \R^n$ such that
\vsm
\noindent
 (a): $|c|+d+\lambda+ \|(q_1, \dots,q_N)\|_{L^\infty} \not= 0$,
 \vspace{0.2 in}

\noindent
 (b): for each $i =1, \ldots, N$ and a.e. $s \in [0,\bar S ]$
 \vspace{0.05 in}

$- \dot q_i(s)=  q_i(s) \cdot f_{ x_1}( \bar y_i (s),\bar y_{i-1}(s) ) (1-  \underset{j=1}{\overset{N}{\sum}} \, \bar \alpha_{j }(s))$

\hspace{0.5 in} 
$  + q_i(s) \cdot 
g_{x_1}(\bar y_i (s),\bar y_{i-1}(s))\bar\alpha_i (s)+ q_{i+1}(s) \cdot 
g_{x_2}(\bar y_{i+1} (s),\bar y_{i}(s)) \bar\alpha_{i+1} (s)$,


\hspace{0.5 in} 
$+  \;q_{i+1}(s) \cdot f_{x_2}( \bar y_{i+1} (s),\bar y_{i}(s) ) (1-  \underset{j=1}{\overset{N}{\sum}} \, \bar \alpha_j (s))$

 \vspace{0.05 in}
 \noindent 
 
where  $q_{N+1}\equiv0$.

 \vspace{0.2 in}

\noindent
 (c): $-q_N(\bar S) \in \lambda \nabla \Psi (\bar y_N (\bar S)) + N_\T(\bar y_N (\bar S)) $ and, for  $i =1, \ldots, N-1$,
 
\hspace{0.2 in}$q_i(\bar S)= q_{i+1}(0)$,
 \vspace{0.2 in}

\noindent
 (d): for  a.e. $s \in [0,\bar S]$,\footnote{Actually, the second  equality involving the maximized Hamiltonian is satisfied for {\em all} $s\in[0,\bar S]$,  by the continuity of the  functions $s  \mapsto  q_i(s)\cdot g(\bar y_i (s), \bar y_{i-1} (s) )$ and $s \mapsto  q_i(s)\cdot f(\bar y_i (s), \bar y_{i-1} (s) )$.}
 $$
 H(\{\bar y_i(s)\},c,d,\{q_i(s)\},\{\bar\alpha_i(s)\}) =
 \underset{\{a_i\} \in \Lambda^{(N)}}{\max}  H(\{\bar y_i(s)\},c,d,\{q_i(s)\},\{a_i\})=0,
 $$
%
%
%
%
%
%
%
%
where
$$
\begin{array}{l}
 H(\{y_i\},c,d,\{q_i\},\{a_i\})=H(y_0,y_1,\dots,y_N,c,d,q_1,\dots,q_N,a_1,\dots,a_N) \\
\qquad:= -c (1-  \sum_{j=1}^{N} a_j ) -d   \sum_{j=1}^{N} a_j \\
\qquad\qquad+ \sum_{j=1}^{N} \, q_j \cdot\Big( f( y_j ,  y_{j-1})(1-  \sum_{j=1}^{N} a_j ) +  g( y_j ,y_{j-1}(s))a_j   \Big).
\end{array}
$$
Notice that  the non-triviality condition  (a) can be replaced   by 
\vsm
\noindent
 (a)$'$: $ \lambda+ \|(q_1, \dots,q_N)\|_{L^\infty} \not= 0$.
\vsm
 \noindent To prove (a)$'$, suppose to the contrary that $(q, \lambda)=(0,0)$ but $|c|+d \not=0$. Then, integrating condition (d) we obtain that  $-c\,\bar\tau(\bar S)-d\,\bar v(\bar S)=-c\,h-d\,\bar v(\bar S)=0$, so 
 $c=-\frac{\bar v(\bar S)}{h}\,d\le0$. However, for  $a_1=\dots=a_N=0$  the maximality condition in (d) implies that $c\ge0$. Hence $c=d=0$ and we get a contradiction (we remember that, if $\bar v(\bar S)<K$, then $d=0$). 


\vspace{0.05 in}

As we shall see, the assertions of Theorem \ref{PMP} are consequences  of relations (a)$'$--(d) concerning the re-parameterized process $(\bar S, \{\bar y_i\}, \{\bar\alpha_i \})$, re-expressed in terms of the original minimizing extended process $(\bar x,\bar \mu, \{\bar w^{r}\}_{ r \in [0,h]})$.

Let the continuous mapping  $\tilde\psi: [0,N\bar S] \rightarrow \R$ and its right inverse $\tilde\sigma:[0,T]\to[0,N\bar S]$ be as in the proof of  Thm.  \ref{Prop_3_2}, when $\{\alpha_i\} = \{\bar \alpha_i\}$.  Precisely, let  $\tilde\psi(s)=(i-1)h +\psi(s-(i-1)\bar S)$ for any $s\in[(i-1)\bar S, i\bar S]$ and $i=1,\dots,N$, where
$
\psi(s) =   \int_0^s (1 -  \sum_{j=1}^{N}  \bar \alpha_j (s') )d s'
$
for  $s \in [0,\bar S]$.
Observe that, in view of Cor. \ref{Cor_inv}   one has $\bar x(t)=\tilde y(\tilde\sigma(t))$ for any $t\in[0,T]$, 
where
\[
\tilde y(s)=
\begin{cases}
\bar y_i(s-(i-1)\bar S) \qquad &\text{if $s\in[(i-1)\bar S,i\bar S]$ for some $i=1,\dots,N$,}\\
\xi_0 \qquad &\text{if $s<0$.}
\end{cases}
\]
Note also that the measure $\bar \mu$ has distribution 
\bel{mu_sigma}
\bar \mu([0,t])= \int_{0}^{\tilde\sigma^+ (t)} \tilde \alpha (s)  ds \qquad \text{for $t\in[0,T]$},
\eeq
where $\tilde\alpha(s)=\bar\alpha_i(s-(i-1)\bar S)$ for any $i=1,\dots,N$ and $s\in[(i-1)\bar S,i\bar S[$.
Moreover, for each $i \in  \{ 1,\dots, N \}$ and $r \in [0,h]$, one has
\begin{equation}
\label{relationx_y}
\left\{
\begin{array}{l}
\bar w_i^r (s) = (\sigma^+ (r) - \sigma^- (r))\,    \bar \alpha_i (\sigma^-(r) + s(\sigma^+ (r) - \sigma^- (r)) )\, \mbox{ a.e.  } s \in [0,1],
\\[1,5ex]
\bar \zeta_i^r (s) = \bar y_i  
(\sigma^-(r) + s(\sigma^+ (r) - \sigma^- (r)) ), \; \mbox{ all } s \in [0,1]\,,
\end{array}
\right.
\end{equation} 
where $\sigma$ denotes the right inverse of $\psi$ above.
Here, the $\bar\zeta^r_i$'s are as in the theorem's statement. 
 For each $i=1,\dots,N$ and $r\in[0,h]$, we set
$
\eta_i^r (s) = q_i (\sigma^-(r) + s(\sigma^+ (r) - \sigma^- (r) )) 
$
for all $s \in [0,1]$.
The $\eta_i^r$'s are the functions used in the theorem statement to calculate the jumps in the costate.   Moreover, 
we introduce the continuous function $\tilde q:[0,N\bar S]\to \R^n$, given by
$\tilde q(s) := q_i(s-(i-1)\bar S)$ for any $s\in [(i-1)\bar S,i\bar S]$ and $i=1,\dots,N$,
and define   $p:[0,+\infty[\to \R^n$ by  $p(t):=\tilde q(\tilde\sigma(t))$ for $t\in[0,T]$ and $p(t)=0$ for $t>T$.
Notice that $p$ is right continuous on $]0,T[$. Since $\tilde\sigma(T)=N\bar S$, our construction together with relation (c)   yields condition (C) in the theorem  statement. 

By the above definitions, it immediately follows that $p(T)=\eta_N^h(1)$, $p((i-1)h)=\eta_i^0(1)$ for any $i=2,\dots,N$ (since $\tilde\sigma((i-1)h)=i\bar S +\sigma^+(0)$), and $p(0)=\eta_1^0(0)$. Moreover, it is straightforward to check that the functions $\eta_i^r$ satisfy \eqref{p_rho}-\eqref{jump_ode2_1}. In particular, for the boundary condition \eqref{jump_ode2_1} we use that $q_{i}(h)=q_{i+1}(0)$ for $i=1,\dots,N-1$.

Condition (A) of the theorem's statement follows by relation (a)$'$ above. Indeed, by (a)$'$ and the very definition of $\tilde q$ we get $\lambda+\|\tilde q\|_{L^\infty} \neq 0$. Hence, if it were $\lambda+\|p\|_{L^\infty}=0$, then $\lambda=0$ and $\tilde q\equiv 0$ on $[0,N\bar S]\setminus D$, as for any $s\in [0,N\bar S]\setminus D$ there exists some $s\in[0,T]$ such that $s=\tilde\sigma(r)$. In the previous relations, $D=\cup_k [\tilde\sigma^-(r_k), \tilde\sigma^+(r_k)]$, where $ \{r_k\}$ is the countable set of jumps of $\bar\mu$. By continuity, $\tilde q\equiv 0$ on the boundary of $[0,N\bar S]\setminus D$, so that $\eta_i^r(0)=\eta_i^r(1)=0$ for any $i=1,\dots,N$ and each $r\in[0,h]$. Accordingly, the linearity of the adjoint equations
yields that $\eta_i^r\equiv0$ for any $i=1,\dots,N$ and any $r\in[0,h]$, which implies $\tilde q \equiv 0$ on $D$. A contradiction.

Let $t\in](i-1)h, ih[$ for some $i=1,\dots,N$, then condition (b) implies
\begin{equation*}
\begin{split}
&p(t) = \tilde q(\tilde\sigma(t)) = \tilde q(\tilde\sigma((i-1)h)) + \int_{\tilde\sigma((i-1)h)}^{\tilde\sigma(t)} \frac{d\tilde q}{ds}  (s) \, ds= \eta_i^0(1) \\
& \qquad - \int_{\tilde\sigma((i-1)h)}^{\tilde\sigma(t)} [\tilde q(s) \cdot f_{x_1}(\tilde y(s), \tilde y(s-\bar S)) + \tilde q(s+\bar S)\cdot f_{x_2}(\tilde y(s+\bar S), \tilde y(s))] \frac{d\tilde\psi}{ds}(s) \, ds \\
&\qquad- \int_{\tilde\sigma((i-1)h)}^{\tilde\sigma(t)} [\tilde q(s) \cdot g_{x_1}(\tilde y(s), \tilde y(s-\bar S))\tilde\alpha(s) \\
&\qquad\qquad\qquad+ \tilde q(s+\bar S)\cdot g_{x_2}(\tilde y(s+\bar S), \tilde y(s))\tilde\alpha(s+\bar S)]  ds,
\end{split}
\end{equation*}
in which $\tilde q\equiv0 $ for $s>N\bar S$.
Employing analogous analysis to that earlier
   used in the proof of  Thm.  \ref{Prop_3_2} part (A) (based on Lemma \ref{lem_3_3}), we deduce   the `advance' differential equation in integral form in condition (B).

\noindent
The second equality in condition (d) tells us:
 for all $s\in[0,\bar S]$ and   $(a_1,\dots,a_N)\in\Gamma^{(N)}$,
\begin{equation*}
\begin{split}
&\Big(-c+\underset{j=1}{\overset{N}{\sum}} q_j(s)\cdot f(\bar y_j(s),\bar y_{j-1}(s))\Big)\Big(1-\underset{j=1}{\overset{N}{\sum}}a_j\Big)\\
&\qquad\qquad\qquad\qquad\qquad\qquad+\underset{j=1}{\overset{N}{\sum}}\Big(-d+ q_j(s)\cdot g(\bar y_j(s),\bar y_{j-1}(s))\Big)\,a_j\le 0.
\end{split}
\end{equation*}
Choosing $(a_1,\dots,a_N)=e_j$ for some element $e_j$ of the canonical basis of $\R^N$, we get  
\bel{diseq_imp}
 -d+ q_j(s)\cdot g(\bar y_j(s),\bar y_{j-1}(s)) \le 0 \quad \forall s\in[0,\bar S], \ \ \forall j=1,\dots,N,
\eeq
while the choice  $(a_1,\dots,a_N)=0$ yields
$
-c+\underset{j=1}{\overset{N}{\sum}} q_j(s)\cdot f(\bar y_j(s),\bar y_{j-1}(s))\le0$ for all $s\in[0,\bar S]$.
Take $s\in[0,\bar S]$ and  let $(\bar a_1,\dots,\bar a_N)\in\Lambda^{(N)}$ be a point where the maximum in condition (d) is obtained. Then,     $\max_{j=1,\dots,N}\left\{ -d+ q_j(s)\cdot g(\bar y_j(s),\bar y_{j-1}(s))\right\}=0$  if  $(\bar a_1,\dots,\bar a_N)\ne0$ (actually, $-d+ q_i(s)\cdot g(\bar y_i(s),\bar y_{i-1}(s))=0$ for each $i$ such that $\bar a_i>0$),   $-c+\underset{j=1}{\overset{N}{\sum}} q_j(s)\cdot f(\bar y_j(s),\bar y_{j-1}(s))=0$  if $\underset{j=1}{\overset{N}{\sum}}\bar a_j<1$, and, for any  $s\in[0,\bar S]$,
\bel{max_i}
\max\Big\{\underset{j=1}{\overset{N}{\sum}} q_j(s)\cdot f(\bar y_j(s),\bar y_{j-1}(s))-c\,,\, \max_{j=1,\dots,N} q_j(s)\cdot g(\bar y_j(s),\bar y_{j-1}(s))-d\Big\}=0.
\eeq
It follows from the above relations and the very definition of $\{\bar\alpha_i\}$ that there exists a measurable set $\mathcal{N}$ of zero Lebesgue measure, such that, for all $s\in[0,\bar S]\setminus \mathcal{N}$:
\begin{gather}
\underset{j=1}{\overset{N}{\sum}} \bar \alpha_j (s) >0 \ \implies \ \max_{j=1,\dots,N} q_j(s)\cdot g(\bar y_j (s),\bar y_{j-1} (s) )-d = 0, \label{supp_ineq_a_1} \\
\bar \alpha_i (s) >0 \ \text{for some  $i \in \{1,\dots,N\}$} \ \implies \ q_i(s)\cdot g(\bar y_i (s),\bar y_{i-1} (s) )-d = 0,  \label{supp_ineq_a_2}\\
\underset{j=1}{\overset{N}{\sum}} \bar \alpha_j (s)<1 \ \implies \  \underset{j=1}{\overset{N}{\sum}}  q_j (s) \cdot f(\bar y_j (s), \bar y_{j-1} (s)) -c= 0. \label{supp_ineq_a_3}
\end{gather}
Expressing  \eqref{diseq_imp} and \eqref{supp_ineq_a_2}
 in terms of $\tilde q$ and $\tilde y$,
  we obtain   
\bel{tilde_g<0}
 \tilde q(s) \cdot g(\tilde y (s), \tilde y(s-\bar S)) -d\le0 \quad \text{for all } s\in[0,N\bar S],
\eeq
and
\bel{tilde_g=0}
 \tilde q(s) \cdot g(\tilde y (s), \tilde y(s-\bar S)) -d=0 \qquad \text{for a.e.  } s\in[0,N\bar S] \  \text{such that} \ \tilde\alpha(s)>0,
\eeq
 respectively. Inserting into these relations  $s=\tilde\sigma(t)$  for any $t\in[0,T]$, from \eqref{tilde_g<0} we derive immediately  condition (D),\,(ii). As for  (D),\,(iii), take any $t \in \mbox{supp}\, (\bar\mu )$. Set $s:=\tilde\sigma(t)$. From \eqref{mu_sigma} it follows that the set
$
 \{s' \in [0,N\bar S] \, :\,  \tilde\alpha (s')>0 \mbox{ and } |s'- s|\leq \epsilon \} 
$
has positive Lebesgue measure  for every $\epsilon >0$. So,    by (\ref{tilde_g=0})  and `continuity' we get
 that $\tilde q(s)\cdot g(\tilde y (s), \tilde y (s-\bar S) )-d=0$.   Since  $s:=\tilde\sigma(t)$, this yields condition (D),\,(iii). 


Let $r \in ]0,h[$ and write $s = \sigma^+(r)$.  Then $s < \bar S $ and consequently the set 
$
 \{s' \in [0, \bar S]\, :\, \sum_i \bar \alpha_i (s') <  1\mbox{ and } |s'- s|\leq \epsilon \} 
$
has positive Lebesgue measure for all $\epsilon >0$. Using (\ref{supp_ineq_a_3}), we can show that, for some point arbitrarily close to $s$ and therefore, by continuity,   at $s$ itself,
\bel{qperf}
\underset{j=1}{\overset{N}{\sum}}  \tilde q(s+(j-1)\bar S) \cdot f(\tilde y (s+(j-1)\bar S), \tilde y(s +(j-2)\bar S)) -c= 0 \,.
\eeq
Since $s=\sigma^+(r)$ and $\tilde\sigma^+(r+(j-1)h)=\sigma^+(r)+(j-1)\bar S$, this implies condition (D),\,(i).  Notice that if $\underset{j=1}{\overset{N}{\sum}}\int_0^1\bar w^0_j(s)\,ds=\sigma^+(0)-\sigma^-(0)=\sigma^+(0)=0$, then the equality in condition (D),\,(i) is valid also at $r=0$.  Similarly, it is satisfied also at $r=h$ whenever both  $\underset{j=1}{\overset{N}{\sum}}\int_0^1\bar w^0_j(s)\,ds=0$ and  $\underset{j=1}{\overset{N}{\sum}}\int_0^1\bar w^h_j(s)\,ds=\sigma^+(h)-\sigma^-(h)=\bar S-\sigma^-(h)=0$. In this case indeed, \eqref{qperf} is satisfied for $s=\bar S=\sigma^-(h)$ and $\tilde\sigma(ih)=i\bar S$   by \eqref{2_2_0ih}.

Now take any $r\in [0,h]$ such that  $\sum_{i=1}^N\int_0^1\bar w^r_i(s)\,ds>0$.  Then $\sigma^+ (r)> \sigma^- (r)$ and $\sum_j \bar \alpha_j (s') =1$, for a.e. $s' \in [\sigma^- (r), \sigma^+ (r)]$. It follows from (\ref{max_i})--(\ref{supp_ineq_a_3}) and `continuity' that 
$
\sum_j\, q_j(s)\cdot f(\bar y_j (s), \bar y_{j-1}(s))-c \leq  \max_j\, q_j(s)\cdot g( \bar y_j (s), \bar y_{j-1}(s))-d=0  
$
for all  $s \in [\sigma^-  (r),\sigma^+ (r)]$ and, for any $i$  such that $\bar \alpha_i (s) >0$,
$
 \,   q_i(s)\cdot g(\bar y_i (s), \bar y_{i-1}(s))= \max_j\, q_j (s)\cdot g(\bar y_j (s), \bar y_{j-1}(s)
$
for  a.e.   $s \in [\sigma^-  (r),\sigma^+ (r)]$.
It follows now from (\ref{relationx_y}) and the very definition of $\eta_i^r$ that condition (D), (iv) is satisfied.
\end{proof}

\end{document}